\documentclass{article}
\usepackage{amsmath}
\numberwithin{equation}{section}
\usepackage{caption}
\usepackage{subcaption}
\usepackage{cite}
\usepackage[english]{babel}
\usepackage[utf8]{inputenc}
\usepackage{amssymb, xcolor, amsthm} 
\usepackage{tikz}
\usepackage{epsfig}
\usepackage{graphicx}
\usepackage{mathtools}
\usepackage{datetime}
\usepackage{hyperref}
\usepackage[utf8]{inputenc}
\usepackage[english]{babel}
\usepackage{amsfonts}
\usepackage{blindtext}
\usepackage[english]{babel}
\newtheorem{theorem}{Theorem}[section]
\newtheorem{corollary}{Corollary}[section]

\newtheorem{proposition}{Proposition}[section]
\newtheorem{remark}{Remark}[section]
\theoremstyle{definition}
\newtheorem{definition}{Definition}[section]
\newtheorem{Example}{Example}[section]

\DeclareMathOperator\Span{span}
\DeclareMathOperator\Lip{Lip}

\newcommand{\wh}[1]{\widehat{#1}}

\newcommand{\n}[1]{{\left\|{#1}\right\|}}
\newcommand{\be}{\begin{equation}}
\newcommand{\ee}{\end{equation}}

\def\N{{\mathbb{N}}}

\def\R{{\mathbb{R}}}

\def\f{{\mathcal{F}}}

\usepackage{amssymb}
\usepackage{amscd}
\usepackage{mathrsfs}  
\usepackage{amsbsy}
\usepackage{amsthm}

\title{Approximation by Quantum Meyer K\"onig and Zeller-Fractal Functions}
\author{}
\date{}

\begin{document}
\begin{center}
{\Large \bf {Approximation by Quantum Meyer-K\"onig-Zeller Fractal Functions}}
\end{center}
\begin{center}
{{ D. Kumar,  A. K. B. Chand, P. R. Massopust }}\\
Department of Mathematics\\
Indian Institute of Technology Madras\\
Chennai - 600036, India, \\
Centre of Mathematics\\
Technical University of Munich (TUM)\\
85748 Garching b. M\"unchen, Germany\\
Email: deependra030794@gmail.com; chand@iitm.ac.in; massopust@ma.tum.de
\end{center}

\begin{abstract}
In this paper, a novel class of quantum fractal functions is introduced based on the Meyer-K\"onig-Zeller operator $M_{q,n}$.  These quantum Meyer-K\"onig-Zeller (MKZ) fractal functions employ $M_{q,n} f$ as the base function in the iterated function system for $\alpha$-fractal functions. For $f\in C(I)$, $I$ closed in $\mathbb{R}$, it is shown that there exists a sequence of  quantum MKZ fractal functions $\{f^{(q_n,\alpha)}_n\}_{n=0}^{\infty}$ which converges uniformly to $f$ without altering the scaling function $\alpha$. The shape of $f^{(q_n,\alpha)}_n$  depends on $q$ as well as the other IFS parameters. For $f,g\in C(I)$ with $g > 0$ or $f\geq g$, we show that there exists a sequence  $\{f^{(q_n,\alpha)}_n\}_{n=0}^{\infty}$ with $f^{(q_n,\alpha)}_n \geq g$ converging to $f$.  Quantum MKZ fractal versions of some classical M\"untz theorems are also presented. For $q=1$, the box dimension and some approximation-theoretic results of MKZ $\alpha$-fractal function are investigated in $C(I)$. Finally, MKZ $\alpha$-fractal functions are studied in $L^p$ spaces with ${p \geq 1}$.
\end{abstract}

\noindent
{\bf Keywords:} Fractal interpolation functions,  quantum Meyer-K\"onig-Zeller operator, smooth quantum fractal functions, constrained approximation,   M\"untz polynomials.
\vskip 2pt\noindent
{\bf AMS Subject Classification (2010):} 28A80. 26A06. 41A05. 41A29. 41A30. 65D05. 65D07.

\section{\bf Introduction}\label{1sec2}
Quantum calculus or $q$-calculus is calculus without the use of limits. This theory has been extensively studied in the fields of approximation theory, special functions, combinatorics, number theory, mechanics, quantum physics, and the theory of relativity.  In 1987, Lupa\c{s} \cite{lupas} constructed the $q$-analogue of Bernstein operators and established convergence estimates and shape preserving properties. In the last three decades,  $q$-extensions of  various results in classical  approximation theory has been proposed by several researchers. For an albeit incomplete list, see, for instance, \cite{Gal_and_Gupta, Gupta1, GH,Ostrovska1,Ostrovska2, Phillips}.

Since classical approximation theory and  $q$-approximation theory dispense with  the approximation  of functions using piecewise smooth functions or infinitely differentiable functions, they are not ideal tools to represent non-differ\-entiable functions such as speech signals,  bio-electric recordings, time series, financial series, or seismic data, to name a few.

Fractal functions  bestow a  constructive approximation theory on irregular functions or functions
whose derivative are non-smooth in nature. Fractal functions easily describe functions
that have some degree of self-similarity at different scales.  Using iterated function systems (IFSs), Barnsley \cite{Barnsley1} introduced the construction of fractal interpolation functions (FIFs) to obtain a mathematical representation of data sets arising from irregular functions. He conceptualized the idea of approximation
of  a continuous function $f$ defined on a real compact interval $I$ by
a family of $\alpha$-fractal functions $f^\alpha$ where $\alpha$ is a set of given or appropriately chosen parameters. We refer the interested reader to the vast literature on fractal functions and fractal interpolation and refer only to 
\cite{PRM,M1,m2,massopust1,m3,m4,Navascues1,Navascues2,Navascues_Chand} as an albeit incomplete list of references as they appertain most closely to the setting considered in this paper.

The choice of a base function $b$ is important in the construction
of $f^\alpha$, even though it is avoided in its notation.
The graph of $f^\alpha$ is typically a fractal set and dimension results
for classes of such fractal functions can be found in, for instance, \cite{Akhtar_Prasad_Navascues,BEHM,bm,Gibert_Massopust,HM1,HM2}.

Shape preserving interpolants play an important role in engineering 
and the applied sciences. The question of shape preserving
aspects of a function $f$ by its fractal perturbation function $f^\alpha$ is answered
affirmatively in \cite{Viswanathan_Chand_Navascues} with a suitable choice for $b$ and $\alpha$.

It is known that an $\alpha$-fractal function $f^\alpha$ of $f$ converges to
 $f$  when  the magnitude of the scaling factors of $f^\alpha$ goes to zero.
 Vijender et. al \cite{VCNS1,VCNS2}  proposed a theory of quantum $\alpha$-fractal functions using Bernstein polynomials associated with $f$ as a base function. They showed that the convergence of a sequence of quantum $\alpha$-fractal functions towards the function $f$ follows from the convergence of the $q$-Bernstein polynomials towards $f$, even if the scaling parameters are non-null. 

In this paper, we propose the use of quantum Meyer-K\"onig-Zeller functions as base functions, i.e., we require that $b=M_{q,n}f$, to construct
a novel sequence of quantum MKZ fractal functions denoted by $f^{\alpha}_{q_n,n}$. It is proved that the sequence $f^{\alpha}_{q_n,n}$  converges to $f$ as $n\to \infty$. However, the magnitude/norm of the scaling functions does not go to 0 when $\{q_n\}_{n=1}^{\infty}$ is a sequence in (0,1] such that $\lim\limits_{n\to \infty} q_n =1$. It is also shown that the shape of $f^{\alpha}_{q,n}$ depends on the scaling functions as well as $0<q\leq 1$. We study the shape preserving aspects of quantum MKZ fractal functions and consider quantum MKZ analogs of two classical M\"untz theorems. The latter approach makes use of so-called quantum MKZ fractal M\"untz polynomials.

Setting $q=1$ in the quantum MKZ-fractal function $f_n^{(q,\alpha)}$, we obtain a novel MKZ $\alpha$-fractal function. Some approximation-theoretic properties and the box dimension for the graph of such MKZ $\alpha$-fractal functions are investigated. Finally, we study the existence of MKZ $\alpha$-fractal functions in $L^p$ spaces with $p\geq 1$ and investigate their approximation-theoretic properties.

\section{Background and Preliminaries} 

In this section, we present the foundations of IFSs and the construction of $\alpha$-fractal functions from a suitable IFS. For more details, the interested reader may want to consult \cite{Barnsley1,Hutchinson,M1,Navascues1}.
     
Let $N\in \N:= \{1,2,3,\ldots\}$ and denote by $\N_N := \{1, \ldots, N\}$ the initial segment of $\N$ of length $N$. An IFS $\mathcal{F} := \{X;w_i : i\in \N_N \}$ is a collection of continuous functions on a complete metric space $(X,d)$. $\mathcal{F}$ is called a hyperbolic IFS if each $w_i$ is contractive on $X$, i.e., its Lipschitz constant
\[
s_i := \Lip (w_i) := \sup_{x,y\in X,\,. x\neq y} \frac{d(w_i(x), w_i(y))}{d(x,y)} < 1.
\]

Let $\mathcal{H}(X) := \{ A\subseteq X : A  \text{ is non-empty and compact} \}$. The Hausdorff-Pompeiu metric $h$ on $\mathcal{H}(X)$ is defined by
\[
h(A,B) := \max\{d(A,B),d(B,A)\}, 
\]
where $d(A,B) :=\sup \{d(x,B):x \in A\}$ and $d(x,B):=\inf \{d(x,y):y\in B\}$. It is known that if $(X,d)$ is a complete metric space then $(\mathcal{H}(X),h)$ is also a complete metric space, termed the space of fractals in \cite{Barnsley}. 

The Huchinson map $W : \mathcal{H}(X)\to \mathcal{H}(X)$ is defined by  
\[
W(A) := \bigcup_{i=1}^{N} w_i(A), \quad\forall A\in \mathcal{H}(X).
\]
(See, \cite{Hutchinson}.) If the IFS $\f$ is hyperbolic then $W$ is a contraction on $(\mathcal{H}(X),h)$ with contraction factor $s :=\max\limits_{i\in \N_N} |s_i| < 1$. Thus, by the Banach fixed point theorem, there exists a unique $G$ in $\mathcal{H}(X)$ such that 
\[
G=\lim_{m\to \infty}W^{\circ m} (A), \quad\text{ for any $A \in \mathcal{H}(X)$}, 
\]
where $W^{\circ m}$ denotes the $m$-fold composition of $W$ with itself. The fixed point $G$ is called the attractor of or deterministic fractal generated by the hyperbolic IFS $\f$.

Now, consider a set of interpolation points 
\[
\{(x_j,y_j)\in [x_1, x_N]\times\R: -\infty<x_1<...<x_N<+\infty \wedge j\in \mathbb{N}_N\}.
\] 
Let $u_i,i\in \mathbb{N}_{N-1}$, be a set of homeomorphisms from $I:=[x_1,x_N]$ to $I_i:=[x_i,x_{i+1}]$ satisfying 
\begin{equation}\label{1eq1}
u_i(x_1)=x_i,\quad u_i(x_N)=x_{i+1}.
\end{equation}
For $i\in \N_{N-1}$, let $v_i: I\times K\to K$ be a function where $K$ is  a suitable compact subset of $\mathbb{R}$ that contains all the $y_j$, $j \in \mathbb{N}_N$. (The existence of such a set is shown in, i.e., \cite{massopust1}.) Assume that each
$v_i$ is continuous in the first variable and Lipschitz continuous in the second variable with Lipschitz constant $|\alpha _i|<1,i\in \mathbb{N}_{N-1}$, i.e.,
\begin{equation}\label{1eq2}
v_i(x_1,y_1)=y_i, \quad v_i(x_N,y_N)=y_{i+1},
\end{equation}and 
\begin{equation}\label{1eq3}
|v_i(x,y_1) - v_i(x,y_2)| \leq |\alpha _i| |y_1 - y_2|, \quad\forall\, i\in \mathbb{N}_{N-1}. 
\end{equation}
Let $C(I)=\{f:I\to \mathbb{R} : f ~ \text{is continuous on $I$}\}$ and define 
\[
\mathcal{G} :=\{g\in C(I):g(x_1)=y_1 \wedge g(x_N)=y_N\}.
\]
Defining a metric on $ \mathcal{G}$ by $\rho(h,g) :=\max \{ | h(x)-g(x) | : x\in I \}$ for $g,h\in \mathcal{G}$, makes $(\mathcal{G},\rho)$ into a complete metric space. 

Define a Read-Bajraktarevic (RB) operator\cite{massopust1} $T$ on $(\mathcal{G},\rho)$ by
\begin{equation}\label{1eq4}
Tg(x) := \sum_{i=1}^{N-1}v_i(u_i^{-1}(x),g\circ u_i^{-1}(x))\chi_{u_i(I)}(x), \quad x\in I,
\end{equation}
where $\chi_S$ denotes the characteristic or indicator function of a set $S$.

Using the properties of $u_i$ and $v_i$, it is straight forward to verify that $Tg$ is continuous on  $I$. Also,
 \begin{equation}\label{1eq5}
   \rho(Tg,Th)\leq |\alpha|_{\infty} \rho(g,h), 
  \end{equation}
  where $\alpha:=(\alpha_1,...,\alpha_{N-1})$ and $|\alpha|_{\infty} := \max\{|\alpha _i :i\in \mathbb{N}_{N-1} \}< 1 $.
Hence, $T$ is a contractive map on the complete metric space $(\mathcal{G},\rho)$. Therefore, by the Banach fixed point theorem, $T$ possesses a unique fixed point $f^*\in\mathcal{G}$. Consequently, from (\ref{1eq4}), $f^*$ obeys the self-referential functional equation
 \begin{equation}\label{1eq6}
  f^* = \sum_{i=1}^{N-1} v_i(u_i^{-1},f^*\circ u_i^{-1})\chi_{u_i(I)}
 \end{equation}
on $I$. It can be easily verified that $f^*(x_j)= y_j$, $j\in \mathbb{N}_{N-1} $. 

Now, define mappings $w_i : I\times K \to I_i \times K$ by 
\[
w_i(x,y):=(u_i(x),v_i(x,y)), \quad (x,y)\in I\times K, \;\;i\in \mathbb{N}_{N-1}. 
\]
The graph of $G(f^*)$ of $f^*$ is the attractor of the IFS $\mathcal{I} :=\{I\times K;w_i(x,y)=(u_i(x),v_i(x,y)), i\in \mathbb{N}_{N-1}\}$ and satisfies the self-referential set equation 
\begin{equation}\label{1eq7}
G(f^*)= \bigcup _{i\in \mathbb{N}_{N-1}}w_i(G(f^*)).
\end{equation}
In this setting,  $f^*$ is called a fractal interpolation function (FIF) associated with the IFS $\mathcal{I}$. 

It was observed in \cite{Barnsley1, M1,Navascues1} that the concept of FIF may be used to define classes of fractal functions associated with any function $f \in C(I)$ as described in the following.

For this purpose, let $I:= [x_1,x_N]\subset\R$. For a given $f\in C(I)$, consider a partition  $\Delta :=\{ x_1, x_2,......x_N\}$ of $I$ satisfying $x_1<x_2<\ldots <x_N$, and a continuous function $b:I \to \mathbb{R}$ with $b\neq f$ that satisfies the endpoint interpolation conditions $b(x_1)=f(x_1)$ and $b(x_N)= f(x_N)$. 

Choose an $\alpha = (\alpha_1, \ldots, \alpha_{N-1}) \in (-1,1)^{N-1}$. If, for $i\in \mathbb{N}_{N-1}$, we set 
\begin{equation}\label{1eq8}
u_i(x):=a_ix+b_i\quad\text{and}\quad v_i(x,y):=\alpha_i y +f(u_i(x))-\alpha _i b(x),
\end{equation} 
and determine the constants $a_i$ and $b_i$ via the conditions \eqref{1eq1},
then the IFS $\{[x_1,x_N]\times \mathbb{R};w_i(x,y)=(u_i(x),v_i(x,y)), i\in\mathbb{N}_{N-1}\}$ determines an attractor which is the graph of a fractal function $f^{\alpha}_{\Delta,b} =: f^{\alpha}$. The function $f^{\alpha}$ is referred to as  $\alpha$-fractal function for $f$ and may be considered as the fractalization of $f$ (with respect to the scaling vector $\alpha $, the base function $b$, and the partition $\Delta$).  

The function $f^{\alpha}$ is fixed point of the RB operator $T:C_f(I)\to C_f(I)$ defined by\cite{M1}
\begin{equation}\label{1eq9}
 Tg = f + \sum_{i=1}^{N-1}\alpha _i\, (g - b)\circ u_i^{-1}\,\chi_{u_i(I)}
\end{equation} 
on $I$, where $C_f(I)=\{ g\in C(I): g(x_1)=f(x_1) \wedge g(x_N)=f(x_N)\}.$ Consequently,  $f^{\alpha}$  satisfies the self-referential equation 
\begin{equation*}\label{1eq10}
 f^{\alpha}(x)= f(x)+ \sum_{i=1}^{N-1}\alpha _i (f^{\alpha}(u_i^{-1}(x))- b(u_i^{-1}(x)))\chi_{u_i(I)}(x),\quad x\in I. 
\end{equation*}
The fractal dimension of $f^{\alpha}$ depends on the choice of the scaling vector  $\alpha$ and the $a_i$ \cite{BEHM,HM1,Akhtar_Prasad_Navascues}.

To obtain more flexibility in the construction of fractal functions, the constant  scalings $\alpha_i$, $i\in \mathbb{N}_{N-1}$, can be replaced by continuous functions $\alpha_i\in C(I)$ with $\|\alpha \|_{\infty} := \max\{\|\alpha _i\|_{\infty}:i\in \mathbb{N}_{N-1}\}<1$ in the IFS (\ref{1eq8}). Hence, 
\begin{equation}\label{1eq11}
v_i(x,y)=\alpha_i(x) y +f(u_i(x))-\alpha _i(x) b(x),\quad i\in \mathbb{N}_{N-1}.
\end{equation}
The corresponding  $\alpha$-fractal function is then the fixed point of the RB-operator
\begin{equation}\label{1eq12}
 Tg = f + \sum_{i=1}^{N-1} (\alpha _i\circ u_i^{-1})( g - b)\circ u_i^{-1}\chi_{u_i(I)}
\end{equation}
and satisfies a self-referential equation with location-dependent scalings
\begin{equation}\label{1eq13}
 f^{\alpha}(x)= f(x)+ \sum_{i=1}^{N-1}\alpha _i(u_i^{-1}(x))( f^{\alpha}(u_i^{-1}(x))- b(u_i^{-1}(x)))\chi_{u_i(I)}(x),\quad x\in I.
\end{equation}

Using (\ref{1eq13}), it is easy to show that
\begin{equation}\label{1eq14}
\|f^{\alpha}-f\|_{\infty}\leq \frac{\|\alpha\|_{\infty}}{1-\|\alpha\|_{\infty}}\, \|f-b\|_{\infty}.
\end{equation}

The above inequality shows that an $\alpha$-fractal function $f^{\alpha}$ converges uniformly to $f$ if either  $\|\alpha\|_{\infty} \to 0$ or $\|f-b\|_{\infty} \to 0$. In particular, if $b$ is taken to be a sequence of MKZ quantum functions, the  novel class of MKZ $(q,\alpha)$-fractal functions is obtained.

\section{MKZ $(q,\alpha)$-Fractal Functions}\label{1sec3}
We need the following notation from quantum calculus. 
For $q\in(0,1]$ and $k\in \mathbb{N}$, let
\[
[k]_q := \begin{cases} \frac{1-q^k}{1-q}, & q\neq 1;\\ k, & q = 1.
\end{cases}
\]
The $q$-factorial is defined as
\[
[k]_q! := \begin{cases} 
      [k]_q[k-1]_q.....[2]_q[1]_q, & k\in \N; \\
      1, & k=0.
   \end{cases}
\]
By means of the $q$-factorial, the $q$-binomial coefficients are then defined by
\[
\binom{n}{k}_q := \frac{[n]_q!}{[k]_q![n-k]_q!},
\]
for all integers $n\geq k\geq 0$.

Following \cite{Trif, Aral_Gupta_Agarwal, Heping}, we define a sequence of MKZ-functions on $I=[x_1,x_N]$ for $ f \in C(I)$  by
\begin{equation}\label{1eq15}
\begin{split}
& M_{n,q}f(x) := P_{n,q}(x) \sum _{k=0}^{\infty} \binom{n+k}{k}_q \left(\frac{x-x_1}{x_N-x_1}\right)^k f\left(x_1+(x_N-x_1)\frac{[k]_q}{[k+n]_q}\right), \\
& M_{n,q}f(x_N) :=f(x_N),
\end{split}
\end{equation}
with 
\[
P_{n,q} (x) := \dfrac{\prod\limits_{j=0}^n (x_N-x_1- q^j(x-x_1))} {(x_N-x_1)^{n+1}}.
\]
It is easy to verify that
\begin{equation*}\label{1eq16}
 M_{n,q}f(x_1)=f(x_1).
\end{equation*}
If, in (\ref{1eq11}), we take as the base function $b :=M_{n,q}f$, then the corresponding $\alpha$-fractal function 
\[
f^{(q,\alpha)}_{n} := \mathcal{F} ^{(q,\alpha)}_{\Delta,b}(f)
\]
is termed a  $(q,\alpha)$-fractal function or quantum MKZ fractal function associated with $f \in C(I)$. 

Moreover,
  \begin{equation}\label{1eq17}
 f^{(q,\alpha)}_n = f+\sum_{i=1}^{N-1} (\alpha _i \circ u_i^{-1})(f^{(q,\alpha)}_n - M_{n,q}f)\circ u_i^{-1}\chi_{u_i(I)}, \quad\text{on $I$.}
 \end{equation}
 
It follows from (\ref{1eq15}) and (\ref{1eq17}), that the various quantitative and approxima\-tion-theoretic  properties of $(q,\alpha)$-fractal functions $ f^{(q,\alpha)}_n$ depend on the choices for $q$ and the scaling functions $\alpha_i$. 
    
    The graph of a $(q,\alpha)$-fractal functions $ f^{(q,\alpha)}_n$  is constructed via the IFS 
     \begin{equation}\label{1eq18}
  \f_{(q,n)} = \{ I\times \mathbb{R};w_{n,i}(x,y)=(u_i(x),v_{n,i}(x,y)): i\in \mathbb{N}_{N-1}\}, \quad n\in \mathbb{N},
  \end{equation}  
  where $v_{n,i}(x,y) := f(u_i(x))-\alpha _i (x)(y- M_{n,q}f(x)).$\\
  
  The following theorem ensures the convergences of a sequence of quantum MKZ-fractal functions to $f$ in the sup-norm.
\begin{theorem}\label{1th3.1}
Let $f \in C(I)$. Then, there exists a sequence of quantum MKZ-fractal functions  $\{f^{(q_n,\alpha)}_n\}_{n=0}^{\infty}$ that converges uniformly to $f$ on $I$, where $\{q_n\}_{n=1}^{\infty}$ is  a sequence in $(0,1]$ with $\lim\limits_{n\to \infty} q_n =1$  and $f^{(q_n,\alpha)}_n$ is the fractal function corresponding to the IFS $\f_{(q_n,n)}$ defined in (\ref{1eq18}). Furthermore,   for each integer  $n \geq 3$, we have 
\[ 
\|f^{(q_n,\alpha)}_n-f\|_{\infty} \leq \tfrac{5}{2}\omega \left(f,\tfrac{1}{\sqrt{[n]_{q_n}}}\right) \frac{\|\alpha\|_{\infty}}{1-\|\alpha\|_{\infty}},
\]
where $\omega$ denotes the modulus of continuity of $f$.
\end{theorem}
\begin{proof}
   Let $f^{(q_n,\alpha)},n\in \mathbb{N}$, be the quantum MKZ-fractal function corresponding to $f$. From (\ref{1eq14}), we obtain 
        \begin{equation}\label{1eq19}
       \|f^{(q_n,\alpha)}_n-f\|_{\infty} \leq \frac{\|\alpha\|_{\infty}}{1-\|\alpha\|_{\infty}} \|f-M_{n,q_n}f\|_{\infty}
       \end{equation}
 By \cite[Theorem 2]{Heping} it follows that
 \begin{equation}\label{1eq20}
     \|M_{n,q_n}f-f\|_{\infty} \to 0 \text{~as~ $n\to \infty$},
 \end{equation}
 which implies uniform convergence of $\{f^{(q_n,\alpha)}_n\}_{n=1}^{\infty}$ to $f$ .
 
Applying the result that
\[
\|M_{n,q}f-f\|_{\infty} \leq \tfrac{5}{2}\omega \left(f,\tfrac{1}{\sqrt{[n]_{q}}}\right), \quad n\geq 3,
\]
from \cite[Theorem 2.3]{Trif} to (\ref{1eq19}), we obtain the following estimate:
     \begin{equation}\label{1eq21}
       \|f^{(q_n,\alpha)}_n-f\|_{\infty} \leq \tfrac{5}{2}\omega \left(f,\tfrac{1}{\sqrt{[n]_{q_n}}}\right) \frac{\|\alpha\|_{\infty}}{1-\|\alpha\|_{\infty}},\quad n \geq 3.
\end{equation}
\end{proof}
\begin{Example}
We want to construct a quantum MKZ-fractal function associated with $f(x)= \sin x$ where $x\in[0,1]$. Choose $I :=[0,1]$ and as a partition $\Delta :=\{0, \frac{1}{7}, \frac{2}{7},  \frac{3}{7}, \frac{4}{7}, \frac{5}{7}, \frac{6}{7}, 1\} $. Furthermore, let $\alpha_i(x) := (1+e^{-10x})^{-1}$, $x\in[0,1]$,  $i\in \mathbb{N}_7$. We take  $\{M_{n,q_n}f\}_{n=1}^{\infty}$ as a sequence of base functions.  

The quantum MKZ-fractal functions  are depicted in Figures \ref{1fig1}(a)--(c) and represent the graphs of $f^{(.5,\alpha)}_1$, $f^{(.5,\alpha)}_{50}$, and $f^{(.9,\alpha)}_{50}$, respectively, at the second level of iteration. 
Figure \ref{1fig1}(b) and Figure \ref{1fig1}(c) show  the  effect of $q$ on the quantum MKZ-fractal function.

Figures. \ref{1fig1}(a) and \ref{1fig1}(c) ensure that the fractal function $f^{(.9,\alpha)}_{50}$ provides better approximation of $f(x)=\sin x$, $x\in[0,1]$ than $f^{(.5,\alpha)}_1$. From $f^{(.5,\alpha)}_1$ and $f^{(.9,\alpha)}_{50}$, we observe that these two functions do not have  the same irregularity even when their scaling functions are the same. Note that $f^{(.5,\alpha)}_1$ exhibits irregularities on all scales, whereas $f^{(.9,\alpha)}_{50}$ exhibits irregularities on small scales. Figure \ref{1fig1}(d)  is the blow-up of a small part of $f^{(.9,\alpha)}_{50}$ showing irregularities of $f^{(.9,\alpha)}_{50}$ on small scales.
One reason for these irregularities is that the scaling function $\alpha$ does not satisfy the inequalities $\|\alpha_i\|_{\infty}<\frac{ a_i}{2}, i\in \mathbb{N}_7$\cite[Theorem 3.1]{Viswanathan_Navascues_Chand}.  
\end{Example}
\begin{figure}[ht!]
\centering
\begin{minipage}{0.48\textwidth}
\epsfig{file=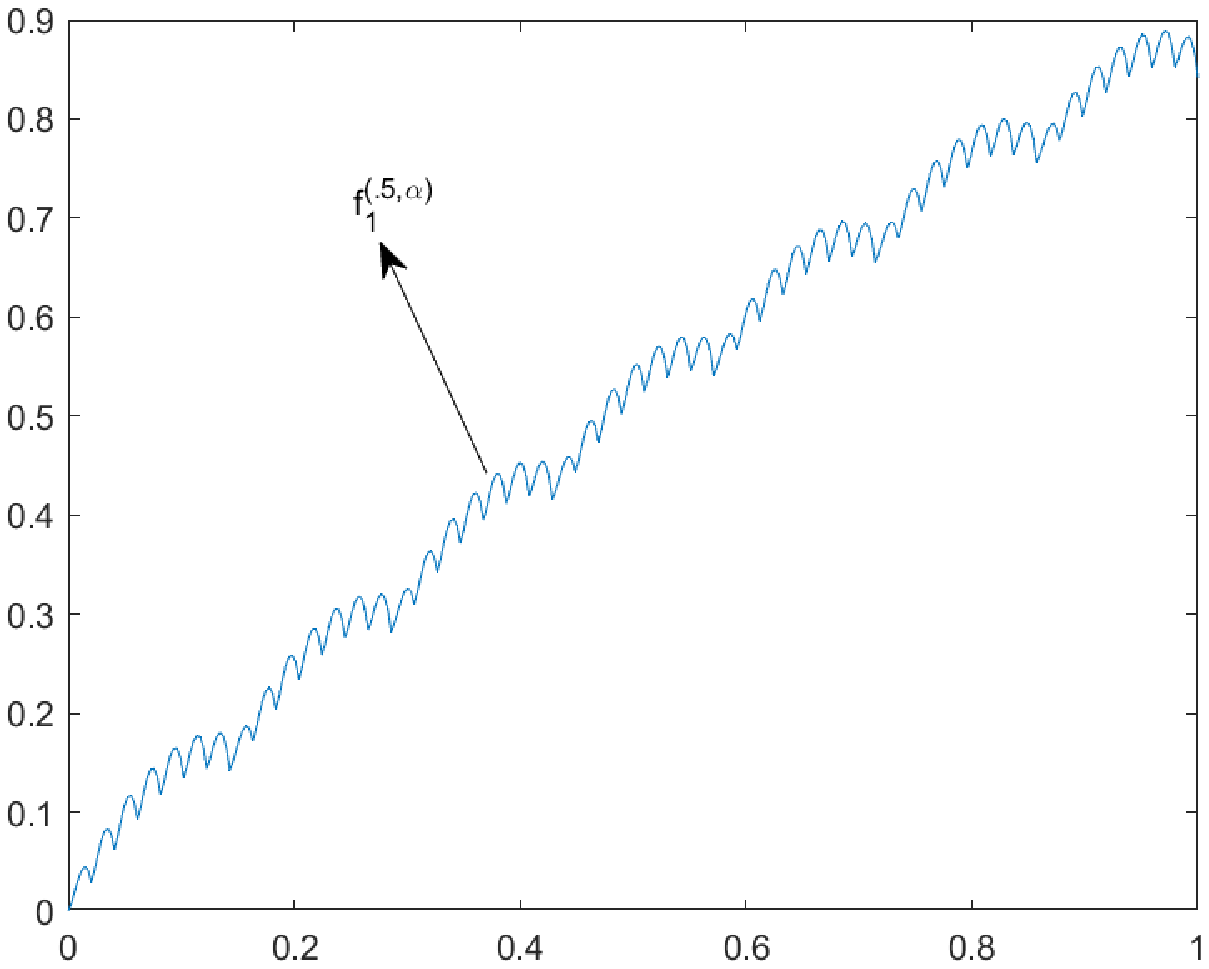,scale=0.4}\\
\centering{(a) $q = 0.5$, $n=1$}
\end{minipage}\hfill
\begin{minipage}{0.48\textwidth}
\epsfig{file=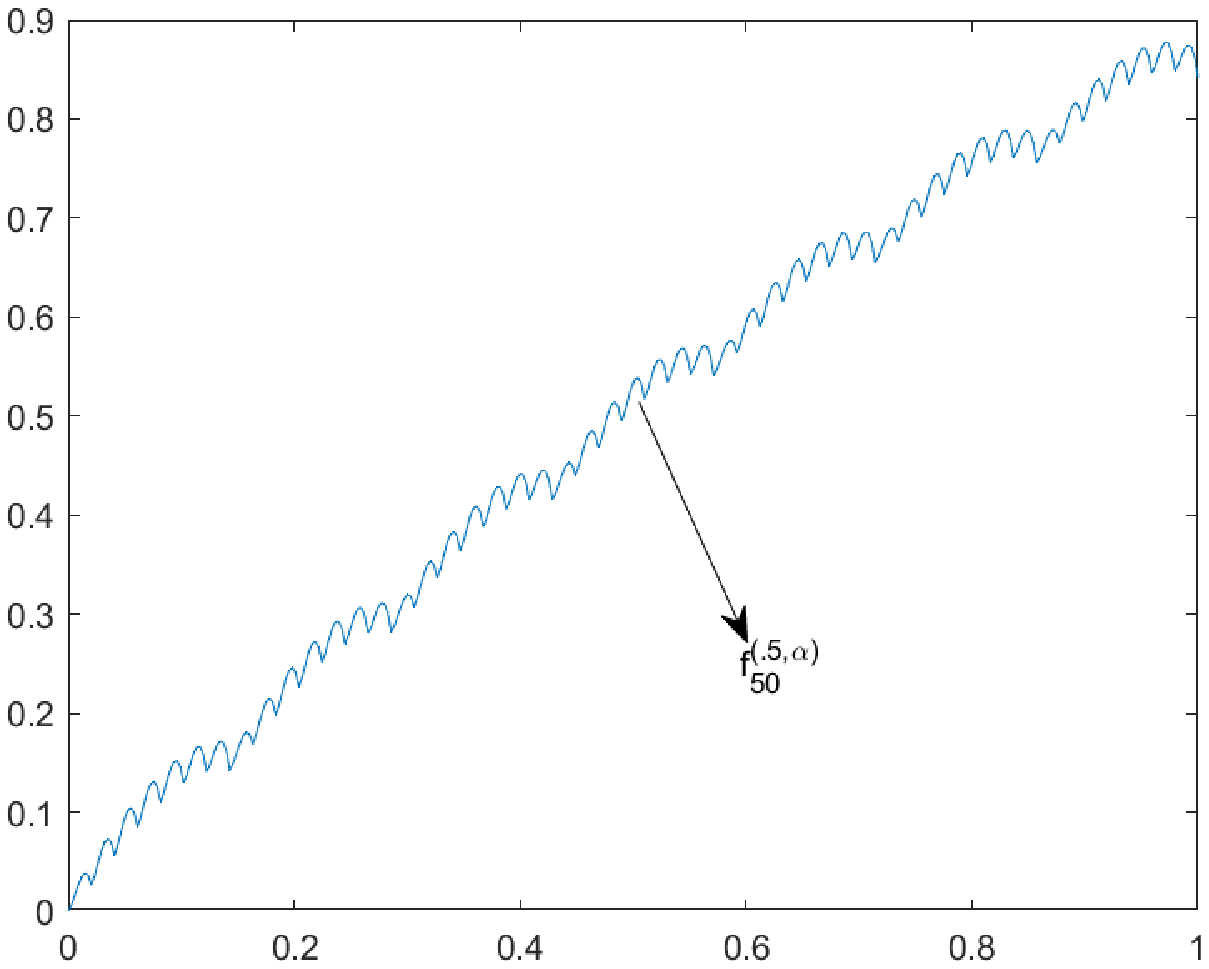,scale=0.4}\\
\centering{(b) $q = 0.5$, $n=50$, Effect of $n$}
\end{minipage}\hfill
\centering
\begin{minipage}{0.48\textwidth} 
\epsfig{file=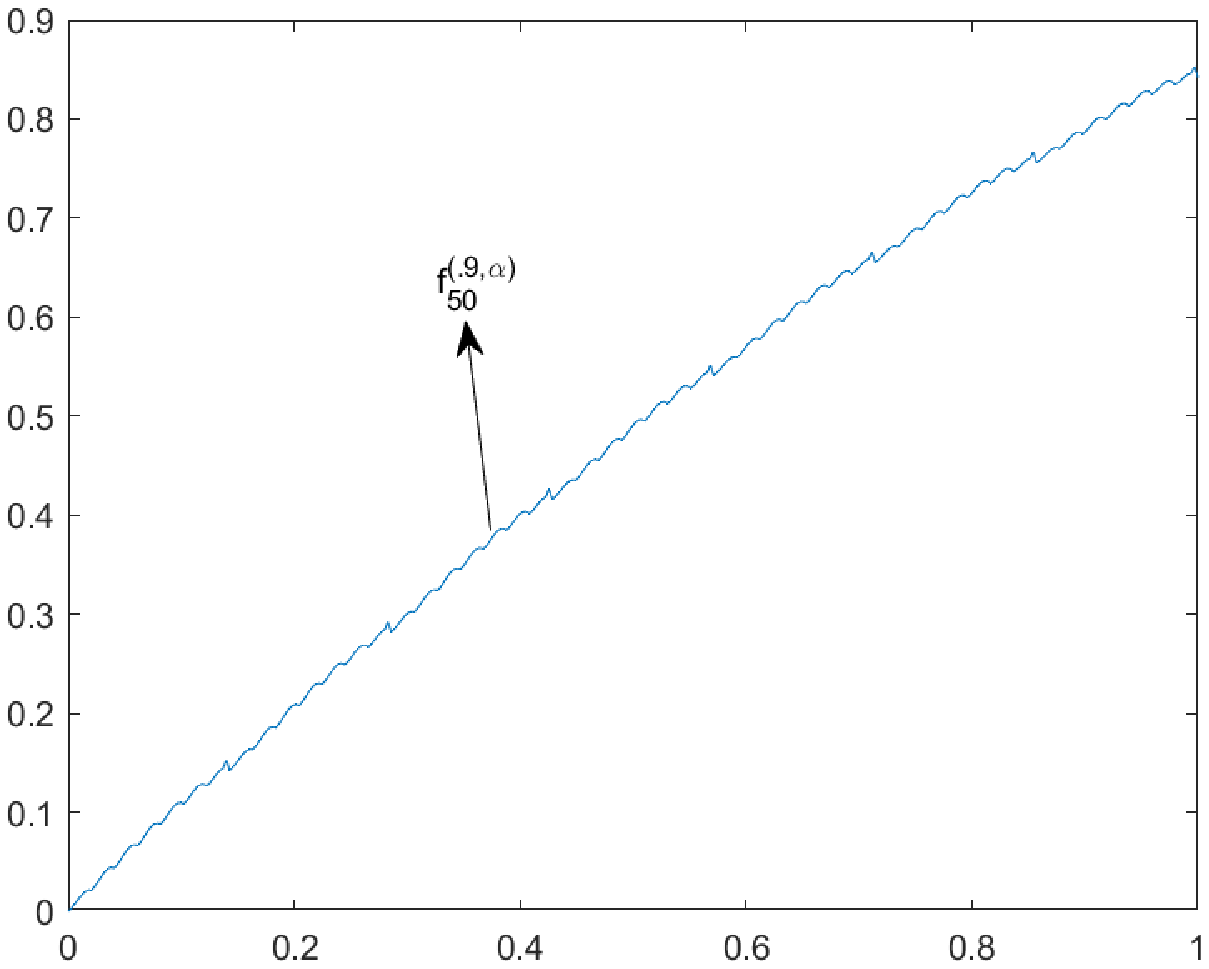,scale=0.4}\\
\centering{(c) $ q=0.9$, $n=50$, Effect of $q$ }
\end{minipage}\hfill
\begin{minipage}{0.48\textwidth}
\epsfig{file=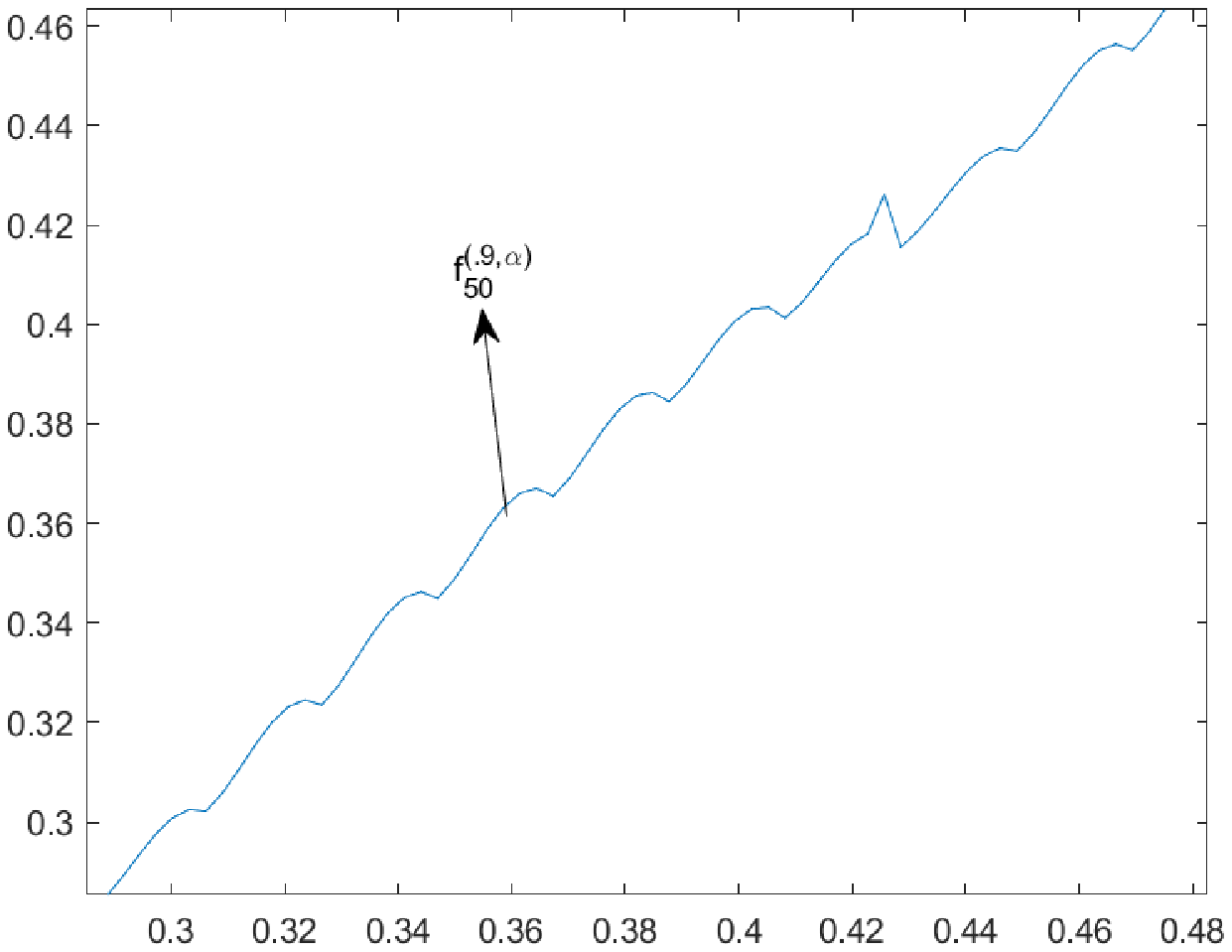,scale=0.4}\\
\centering{(d) Magnified version of part of (c) }
\end{minipage}\hfill
\caption{MKZ-fractal functions of $\sin x$.}
\label{1fig1}
\end{figure}

\begin{theorem}\label{1th3.2}
Let $C(I)$ be endowed with the sup-norm. For every $n\in \mathbb{N}$, the
$(q,\alpha)$-operator $\mathcal{F}_n^{(q,\alpha)}:C(I) \to C(I)$, $\mathcal{F}_n^{(q,\alpha)}(f)=f^{(q,\alpha)}_n$, is bounded and linear.
\end{theorem}
\begin{proof}
We know from \cite{Trif} that $M_{n,q}$ is a positive linear operator. Further, it is known that  $M_{n,q}e_0(x)= e_0(x)$, where $e_0(x)\equiv 1$. Then,
$$-\|f\|_{\infty} e_0 \leq f \leq \|f\|_{\infty}e_0 \quad\implies\quad -\|f\|_{\infty} M_{n,q} e_0 \leq M_{n,q}f \leq \|f\|_{\infty}M_{n,q}e_0.$$ Thus, $\|M_{n,q}f\|_{\infty} \leq \|f\|_{\infty}M_{n,q}e_0=\|f\|_{\infty}.$ Hence $M_{n,q}$ is a bounded operator.  By  reference \cite{Navascues_Chand},  $\mathcal{F}_n^{(q,\alpha)}$ is a linear and bounded operator.
\end{proof} 
 \section{Constrained Quantum MKZ-Fractal Approximation}\label{1sec4}
 In this section, we study constrained approximation by quantum MKZ-fractal functions.
 \begin{theorem}\label{1th4.1}
 Let $f\in C(I)$ and $f\geq 0$ on $I$. Let $\Delta=\{x_1,...,x_N\}$ be a partition of I satisfying the condition $x_1<....<x_N$, and let $\{q_n\}_{n=1}^{\infty}$ be sequence in (0,1] such that $\lim\limits_{n\to \infty} q_n =1$. Then, the sequence $\{\f_{(q_n,n)}\}_{n=1}^{\infty}$ of IFSs determines a sequence $\{f^{(q_n,\alpha)}_n\}_{n=1}^{\infty}$ of non-negative quantum MKZ-fractal functions that converges uniformly to $f$ if the scaling functions $\alpha_i(x)$ are chosen to satisfy the following two conditions:
 \begin{enumerate}
\item  $\|\alpha\|_{\infty}<1$;
\item For all   $i\in \mathbb{N}_{N-1}$,
 \begin{align}\label{1eq22}
\max\left\{ \frac{-\phi(f,i)}{C_n-\phi_n},\right. & \left. -\frac{C_n-\Phi(f,i)}{\Phi_n}\right\}  \leq \alpha_i(x) &\nonumber\\
& \leq \min \left\{ \frac{\phi(f,i)}{\Phi_n},\frac{C_n-\Phi(f,i)}{C_n-\phi_n}\right\}, \quad x\in I.
 \end{align}
Here, we set $\phi(f,i) :=\min\limits_{x\in I}f(u_i(x))$, $\Phi(f,i) :=\max\limits_{x\in I}f(u_i(x))$, $\phi_n :=\min\limits_{x\in I}M_{n,q_n}f(x)$,  and $\Phi_n :=\max\limits_{x\in I}M_{n,q_n}f(x)$. $C_n$ denotes a positive real number strictly larger than $\max\{\phi_n, \|f\|_{\infty}\}$.
 \end{enumerate}
 \end{theorem} 
\begin{proof}
By Theorem \ref{1th3.1}, there exists a sequence $\{f^{(q_n,\alpha)}_n\}_{n=1}^{\infty}$ of quantum MKZ-fractal functions that converges to $f$. Now suppose  that 
    $f\in C(I)$ and $f\geq0$ on $I$. It is known (cf. for instance, \cite{Trif}) that $M_{n,q}$ is positive linear operator and thus $M_{n,q}f\geq 0$ on $I$. This implies that $\Phi_n$ is positive. We have that
\begin{equation}\label{1eq23}
    \begin{split}
        f^{(q_n,\alpha)}_n(u_i(x)) & =v_{n,i}(x, f^{(q_n,\alpha)}_n(x))\\ &=f(u_i(x))+\alpha_i(x)(f^{(q_n,\alpha)}_n(x)-M_{n,q_n}f(x)),\quad  x\in I.
        \end{split}
\end{equation}
Clearly, $v_{n,i}(x,f^{(q_n,\alpha)}_n(x)) \in [0,C_n], i\in \mathbb{N}_{N-1}$, iff $ f^{(q_n,\alpha)}_n(u_i(x))\in[0,C_n]$ for all $x\in I$. Note that $I$ is attractor of the IFS $\{I; u_i(x), i\in \mathbb{N}\}$. Since $f^{(q_n,\alpha)}_n$ is defined recursively, we only need to show that $ f^{(q_n,\alpha)}_n(u_i(x)) \geq0$ whenever $ f^{(q_n,\alpha)}_n(x)\geq 0$ using suitable restrictions on the functions $\alpha_i$. 

To this end, suppose $(x,y)\in I\times[0,C_n]$ and $\alpha_i$, $i\in \mathbb{N}_{N-1}$, is such that $|\alpha_i(x)|<1$, for all $x\in I$. Now, there are two cases: \\

\noindent
\textbf{Case 1:} $ 0\leq \alpha_i(x)<1$, for all $x\in I$. 

Then $ 0\leq y\leq C_n$ yields
$q_{n,i}\leq \alpha_i(x)y+q_{n,i}\leq C_n\alpha_i(x)+q_{n,i}$. Therefore, 
\[
0\leq v_{n,i}(x,y)=\alpha_i(x)y+q_{n,i}\leq C_n, \quad i\in \mathbb{N}_{N-1},
\] 
is true for all $(x,y)\in I\times[0,C_n]$ if
\begin{equation}\label{1eq24}
\begin{split}
&f(u_i(x))-\alpha_i(x)M_{n,q_n}(f,x)\geq0,\\
&f(u_i(x))-\alpha_i(x)M_{n,q_n}(f,x)\leq C_n(1-\alpha_i(x)), \quad x\in I.
\end{split}
\end{equation}
As $f(u_i(x))\geq \phi(f,i)$ and $M_{n,q_n}(f,x)\leq \Phi_n,$ it is easy to verify that  $f(u_i(x))-\alpha_i(x)M_{n,q_n}(f,x) \geq 0$ if $\phi(f,i)-\alpha_i(x)\Phi_n\geq0$ which is equivalent to the condition 
\[
0\leq \alpha_i(x) \leq \frac{\phi(f,i)}{\Phi_n}. 
\]
Similarly, using 
$f(u_i(x))\leq \Phi(f,i)$ and $M_{n,q_n}(f,x)\geq \phi_n$,  the second inequality in (\ref{1eq24}) is true, whenever $\Phi(f,i)-\alpha_i(x) \phi_n \leq C_n(1-\alpha_i(x))$  which is equivalent to 
\[
0\leq \alpha_i(x)\leq \frac{C_n-\Phi(f,i)}{C_n-\phi_n}.
\]  
Combining, these two subcases, we  obtain that $v_{n,i}(x,y) \in [0,C_n]$, $i\in \mathbb{N}_{N-1}$, for all $(x,y)\in I\times[0,C_n]$ if 
\[
0\leq \alpha_i(x) \leq \min \left\{ \frac{\phi(f,i)}{\Phi_n},\frac{C_n-\Phi(f,i)}{C_n-\phi_n}\right\}.
\]

\noindent\textbf{Case 2}: $ -1<\alpha_i(x)\leq0$, for all $x\in I$. 

Then $ 0\leq y\leq C_n$ yields 
$C_n\alpha_i(x)+q_{n,i}\leq \alpha_i(x)y+q_{n,i}\leq q_{n,i}$. Hence,
\[
0\leq v_{n,i}(x,y)=\alpha_i(x)y+q_{n,i}\leq C_n, \quad i\in \mathbb{N}_{N-1},
\]
is valid for all $(x,y)\in I\times[0,C_n]$ whenever
\begin{equation}\label{1eq25}
\begin{split}
& f(u_i(x))-\alpha_i(x)M_{n,q_n}f(x)\leq C_n,\\ 
&C_n\alpha_i(x)+ f(u_i(x))-\alpha_i(x)M_{n,q_n}f(x)\geq0, \quad x\in I.
\end{split}
\end{equation}
As $f(u_i(x))\leq \Phi(f,i)$ and $M_{n,q_n}(f,x)\leq \Phi_n$,  then from the first inequality in (\ref{1eq25}), we obtain $f(u_i(x))-\alpha_i(x)M_{n,q_n}(f,x)\leq \Phi(f,i)-\alpha_i(x)\Phi_n\leq C_n$. Hence, 
\[
\alpha_i(x) \geq -\frac{C_n-\Phi(f,i)}{\Phi_n}. 
\]
Again, due to the fact $M_{n,q_n}f(x)\geq \phi_n$ and  $f(u_i(x))\geq \phi(f,i)$, we observe that second inequality in (\ref{1eq25}) holds if 
\[
\alpha_i(x)\geq \frac{-\phi(f,i)}{C_n-\phi_n}. 
\]
Combining these two results, we conclude that 
$v_{n,i}(x,y) \in [0,C_n]$, $i\in \mathbb{N}_{N-1}$, for all $(x,y)\in I\times[0,C_n]$ if  \[
\max\left\{ \frac{-\phi(f,i)}{C_n-\phi_n},-\frac{C_n-\Phi(f,i)}{\Phi_n}\right\}\leq \alpha_i(x)\leq 0. 
\]
Both cases yield the desired restrictions on the functions $\alpha_i$ in (\ref{1eq22}).
\end{proof}
From Theorem \ref{1th4.1}, it is found that for  every continuous function $f$ on $I$ with $f\geq 0$ on $I$, there exists a sequence of non-negative quantum MKZ-fractal functions which converges to $f$ in the sup-norm.

\begin{Example}
Here, we give an example to illustrate Theorem \ref{1th4.1}. Let $I:=[0,1]$ and let $f:I\to [0,2]$, $x\mapsto \sin(\pi x)+1$. Further, let $\Delta :=\{0, \frac{1}{3}, \frac{2}{3}, 1\}$, and define $\alpha_1(x) :=0.1298/(1+\exp(-10x))$, $\alpha_2(x):=0.100/(1+\exp(-10x))$,  and $\alpha_3(x):=0.2168/(1+\exp(-10x))$. 

Assume that $q_n:=\frac{2}{\pi}\arctan(n)$, $n\in \mathbb{N}$ . Then, the scaling functions $\alpha_i$ and the sequence $\{q_n\}_{n=1}^{\infty}$ fulfill the conditions stated in Theorem \ref{1th4.1}. 
In Figure \ref{2fig} (a), the fractal quantum MKZ-fractal function $f^{(q_2,\alpha)}_{2}$ is shown and provides a positive approximation for  $f$. 

If we choose  $\alpha_1(x) :=0.7$, $\alpha_2(x) :=-0.9$, and $\alpha_3(x) :=0.9$ instead, then $\alpha $ is not consistent with the condition given in (\ref{1eq22}). From Fig. \ref{2fig} (b),  we observe  that the MKZ $(q,\alpha)$-fractal function $f^{(q_2,\alpha)}_{2}$ is non-positive in nature for this choice of $\alpha$. 
\end{Example}
\begin{figure}[ht!]
\centering
\begin{minipage}{0.48\textwidth}\label{1fig1a}
\epsfig{file=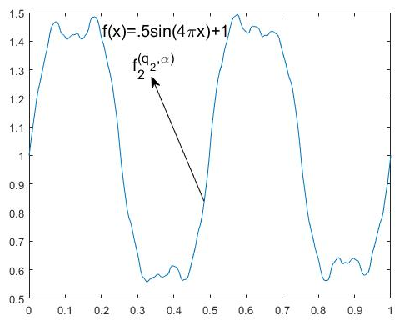,scale=0.9}\\
\centering{(a) Positive quantum  MKZ-fractal function }
\end{minipage}\hfill
\begin{minipage}{0.48\textwidth}\label{1fig1b}
\epsfig{file=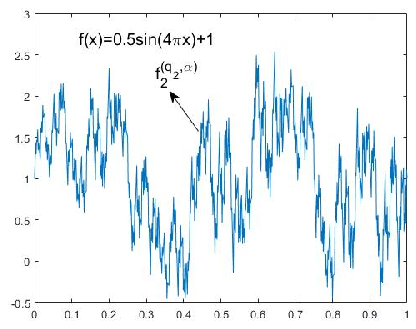,scale=0.9}\\
\centering{(b) Non-positive quantum MKZ-fractal function }
\end{minipage}\hfill
\caption{Positivity of  MKZ-fractal function according to Theorem \ref{1th4.1} }
\label{2fig}
\end{figure}

The following theorem gives the existence of a double sequence of positive quantum MKZ-fractal functions, which converges to $f$ in the sup-norm.

\begin{theorem}\label{1th4.2}
Let $\{f_k\}_{k=1}^{\infty}$ be a sequence of positive functions in $C(I)$ that converges to  $f\in C(I)$. Let $\Delta=\{x_1,...,x_N\}$ be a partition of $I$ satisfying the condition $x_1<...<x_N$ and let $\{q_n\}_{n=1}^{\infty}$ be sequence in (0,1] such that $\lim\limits_{n\to \infty} q_n =1$. 

Suppose that $u_i:I\to I_i$, $i\in \mathbb{N}_{N-1}$, are affine maps of the form $u_i(x)=a_ix+b_i$ satisfying the conditions $u_i(x_1)=x_i$, $u_i(x_N)=x_{i+1},$ and let 
\[
v_{k,n,i}^{\dagger}(x,y):=f_k(u_i(x))+\alpha_i(x)(y-M_{n,q_n}(f_k,x)), \quad i\in \mathbb{N}_{N-1}.
\]
 Let $f^{(q_n,\alpha)}_{k,n}$ be the MKZ-fractal function associated with the IFS 
\[
\f_{k,(q_n,n)}^{\dagger} :=\{I\times K;(u_i(x),v_{k,n,i}^{\dagger}(x,y)),\; i\in \mathbb{N}_{N-1}\}. 
\]
Then, the double sequence of IFSs $\{ \{\f_{k,(q_n,n)}^{\dagger}\}_{n=1}^{\infty} \}_{k=1}^{\infty}$  generates a double sequence $\{\{f^{(q_n,\alpha)}_{k,n}\}_{n=1}^{\infty} \}_{k=1}^{\infty}$ of positive quantum MKZ-fractal functions which converges to $f$ in sup norm provided that all scaling functions $\alpha_i$ obey the conditions:
\begin{enumerate}
\item $\|\alpha_i\|_{\infty}<1$;
\item For all $i\in \mathbb{N}_{N-1}$,
\begin{align}\label{1eq26}
\max & \left\{ \frac{-\phi(f_k,i)}{C_{n,k}^{\dagger} -\phi_{n,k}(f_k)},-\frac{C_{n,k}^{\dagger}-\Phi(f_k,i)}{\Phi_{n,k}(f_k)}\right\} \leq \alpha_i(x) \nonumber\\
& \leq \min \left\{ \frac{\phi(f,i)}{\Phi_{n,k}(f_k)},\frac{C_{n,k}^{\dagger}-\Phi(f_k,i)}{C_{n,k}^{\dagger}-\phi_{n,k}(f_k)}\right\}, \quad x\in I.
 \end{align}
 where $\phi(f_k,i) :=\min\limits_{x\in I}f_k(u_i(x))$,  $\Phi(f_k,i) :=\max\limits_{x\in I}f_k(u_i(x))$, $\phi_{n,k}(f_k)$ $:=\min\limits_{x\in I}M_{n,q}(f_k,x)$, and $\Phi_{n,k}(f_k):=\max\limits_{x\in I}M_{n,q}(f_k,x)$. Here, $C_{n,k}^{\dagger}$ denotes a positive real number strictly greater than $\max\{\phi_{n,k}(f_k), \|f_k\|_{\infty}\}$.
\end{enumerate}
\end{theorem}
\begin{proof}
It follows easily from Theorem \ref{1th4.1} that the MKZ-fractal functions $f^{(q_n,\alpha)}_{k,n}$ are positive on $I$ if the scaling functions $\alpha_i$, $i\in \mathbb{N}_{N-1} $, obey the inequalities in (\ref{1eq26}).

Let $\epsilon >0$. As $\{ f_k\}_{k=1}^{\infty}$ is a sequence of positive functions in $C(I)$ that converges to $f$ in $\|\cdot\|_{\infty}$, there exists a natural number $k_1 \in \mathbb{N}$ such that 
\begin{equation}\label{1eq27}
\|f_k-f \|_{\infty}<\frac{\epsilon}{2},\quad \forall \,k\geq k_1.
\end{equation}
Employing Theorem 2 of \cite{Heping}, we can see that for each $k\in \mathbb{N}$ , $\|M_{n,q_n}f_k-f_k\|_{\infty} \to 0$, as $n\to \infty$. Thus, there exists a $k_2 \in \mathbb{N}$ such that 
\begin{equation}\label{1eq28}
\|M_{n,q_n}f_k-f_k\|_{\infty}<\frac{\epsilon(1- \|\alpha\|_{\infty})}{2\|\alpha\|_{\infty}},~~ \forall n\geq k_2.
\end{equation}
Given that $f^{(q_n,\alpha)}_{k,n}$ is the MKZ-fractal function obtained from the IFS $\f_{k,n}^{\dagger}$, $f^{(q_n,\alpha)}_{k,n}$ satisfies the functional equation
 \begin{align}\label{1eq29}
 f^{(q_n,\alpha)}_{k,n} &=f_k \nonumber\\
 & \qquad + \sum_{i=1}^{N-1}(\alpha_i\circ u_i^{-1})(f^{(q_n,\alpha)}_{k,n}\circ u_i^{-1} - M_{n,q_n} f_k\circ u_i^{-1})\chi_{u_i(I)}
 \end{align}
 on $I$. Eqn. (\ref{1eq29}) implies the inequality 
 \begin{equation}\label{1eq30}
\| f^{(q_n,\alpha)}_{k,n}-f_k\|_{\infty}\leq \frac{\|\alpha\|_{\infty}}{(1- \|\alpha\|_{\infty})}\|f_k-M_{n,q_n}f_k\|_{\infty}.
 \end{equation}
  From (\ref{1eq28}) and (\ref{1eq30}), we obtain 
 \begin{equation}\label{1eq31}
\| f^{(q_n,\alpha)}_{k,n}-f_k\|_{\infty}<\frac{\epsilon}{2}, \quad\forall~ n\geq k_2.
 \end{equation}
 Combining (\ref{1eq27}) and (\ref{1eq31}), shows that 
 for a given $\epsilon>0$, there exists a $k_0 :=\max\{k_1 ,k_2\}$ such that 
 \begin{equation*}\label{1eq33}
      \| f^{(q_n,\alpha)}_{k,n}-f\|_{\infty}< \epsilon, \quad \forall ~k,n \geq k_0,
 \end{equation*}
confirming that the sequence $\{\{f^{(q_n,\alpha)}_{k,n}\}_{n=1}^{\infty}\}_{k=1}^{\infty}$ converges uniformly to $f$.
\end{proof}
The following theorem gives the existence of a one-sided sequential approximation by MKZ-fractal functions.
\begin{theorem}\label{1th4.3}
Let $f,g\in C(I)$ with $f\geq g$ on $I$. Let $\Delta=\{x_1,...,x_N\}$ be a partition of $I$ satisfying the condition $x_1<...<x_N$ and let $\{q_n\}_{n=1}^{\infty}$ be a sequence in $(0,1]$ such that $\lim\limits_{n\to \infty} q_n =1$. For all $n\in \mathbb{N}$, let $f^{(q_n,\alpha)}_n$ denote the MKZ-fractal functions associated with  the IFS $\f_n$. 

Then the sequence of IFSs $\{\f_n\}_{n=1}^{\infty} $ determines a sequence of MKZ-fractal functions $\{f^{(q_n,\alpha)}_n\}_{n=1}^{\infty}$ such that $f^{(q_n,\alpha)}_n\geq g$ on $I$, $n\in \mathbb{N}$, and this sequence converges uniformly to $f$ provided that the scaling function $\alpha$ of each $f^{(q_n,\alpha)}_n$ satisfies the following conditions:
\begin{enumerate}
\item $\|\alpha \|_{\infty}<1$;
\item For each $i\in \mathbb{N}_{N-1}$,
\begin{equation}\label{1eq34}
    0\leq \alpha_i(x)\leq \min\left\{\frac{\phi(f-g,i)}{\Phi_n(f)-\phi(g)},1 \right\},\quad x\in I,
\end{equation}
where $\phi(f-g,i):=\min\limits_{x\in I}(f-g)(u_i(x))$, $\Phi_n(f):=\max\limits_{x\in I}M_{n,q_n}f(x)$, and  $\phi(g):=\min\limits_{x\in I}g(x))$.
\end{enumerate}
\end{theorem}
\begin{proof}
By the construction of an MKZ-fractal function, we observe that $f^{(q_n,\alpha)}_n$ satisfies the following functional equation:
\begin{align}\label{1eq35}
 f^{(q_n,\alpha)}_n(& x) = f(x)\, +\nonumber\\
 & + \sum_{i=1}^{N-1}\alpha _i (u_i^{-1}(x))(f^{(q_n,\alpha)}_n(u_i^{-1}(x))- M_{n,q_n}f(u_i^{-1}(x)))\chi_{u_i(I)}(x), \quad x\in I.   
\end{align}
This functional equation is a rule to get the values of $f^{(q_n,\alpha)}_n$ at $(N-1)^{r+2}+1$ distinct points in $I$ at the $(r+1)$-th iteration using the value of $f^{(q_n,\alpha)}_n$ at $(N-1)^{r+1}+1$ points in $I$ at the $r$-th iteration.
Thus, if we can show that the result is true at the first iteration, then it is true for all
subsequent iterations.

We begin the iteration process at the nodal points $x_i$, $i\in \mathbb{N}$,
where $f^{(q_n,\alpha)}_n\geq$ $g$ as $f^{(q_n,\alpha)}_n$ interpolates $f$ at these nodes and $f\geq g$. Now, we want to verify that  $f^{(q_n,\alpha)}_n\circ u_i\geq g\circ u_i$ on $u_i(I)$. By (\ref{1eq35}), this is equivalent to proving that
\begin{equation}\label{1eq36}
     f\circ u_i + \alpha _if^{(q_n,\alpha)}_n - \alpha _i M_{n,q_n}f - g\circ u_i\geq 0\quad\text{on $u_i(I)$}.
\end{equation}
If we choose the functions $\alpha _i$ to be positive, then the above inequality is true provided that
\[
f\circ u_i + \alpha _i  g - \alpha _i M_{n,q_n}f - g\circ u_i\geq 0 .
\]
The sufficient condition for the validity of the above inequality is
\begin{equation*}\label{1eq37}
    0\leq \alpha_i(x)\leq \min\left\{\frac{\phi(f-g,i)}{\Phi_n(f)-\phi(g)} \right\}, \quad x\in I.
\end{equation*}
Therefore, if the functions $\alpha_i$, $i\in \mathbb{N}_{N-1}$, is chosen according to (\ref{1eq34}), then $f^{(q_n,\alpha)}_n \geq g$ on $I$.
\end{proof} 

\begin{Example}
Let $f(x) :=0.5 \sin(4\pi x)+1$ and  $g(x) :=-0.5(2x-1.1)^2$ be two continuous functions defined on $I:=[0,1]$ and let $\Delta :=\{0, \frac{1}{3}, \frac{2}{3}, 1\}$ be a uniform partition of $I$. Further, let $q_n :=\frac{2}{\pi}\arctan(n)$, $n\in \mathbb{N}$. 

If we take $\alpha_1(x) :=0.3950/(1+\exp(-10x))$, $\alpha_2(x) :=0.3550/(1+\exp(-10x))$, and $\alpha_3(x):=0.2774/(1+\exp(-10x))$,  then the scaling function $\alpha$  satisfies the required conditions (\ref{1eq34}) in Theorem \ref{1th4.3}. 
Fig. \ref{3fig} (a) shows the  MKZ-fractal function $f^{(q_2,\alpha)}_2$ and verifies that $f^{(q_2,\alpha)}_2 \geq g$ on $I$. 

Similarly, one can vary $n$ to construct one-sided approximants $f^{(q_n,\alpha)}_n \geq g$ on $I$. But when $\alpha_1(x) :=0.4$, $\alpha_2(x) :=0.355$, and $\alpha_3(x) :=0.8$, then the scaling function $\alpha$ does not satisfy condition (\ref{1eq38}). In Fig. \ref{3fig} (b), it is shown that for this choice of $\alpha$, the MKZ-fractal function obey $f^{(q_2,\alpha)}_2\not\geq g$ on $I$.
\end{Example}
\begin{figure}[ht!]
\centering
\begin{minipage}{0.5\textwidth}
\epsfig{file=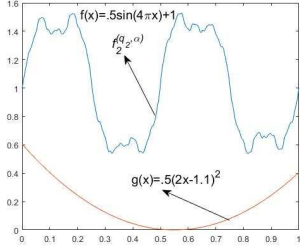,scale=0.99}\\
\centering{(a) $\alpha$  satisfies the sufficient condition (\ref{1eq38}) }
\end{minipage}\hfill
\begin{minipage}{0.5\textwidth}
\epsfig{file=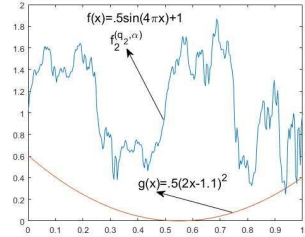,scale=0.99}\\
\centering{(b) $\alpha$ does not satisfy (\ref{1eq38}) }
\end{minipage}\hfill
\caption{One-sided approximantion by MKZ-fractal function according to Theorem \ref{1th4.3} }
\label{3fig}
\end{figure}
\begin{corollary}\label{1coro4.4.1}
Let $f,g \in C(I)$ with $f\geq g$ on $I$. Let $\Delta :=\{x_1,...,x_N\}$ be a partition of $I$ satisfying the condition $x_1<...<x_N$ and let $\{q_n\}_{n=1}^{\infty}$ be a sequence in $(0,1]$ such that $\lim\limits_{n\to \infty} q_n =1$. Then, there exist sequences  $\{f^{(q_n,\alpha)}_n\}_{n=1}^{\infty}$ and $\{g^{(q_n,\alpha)}_n\}_{n=1}^{\infty}$ of $(q,\alpha)$-fractal functions which converge to $f$ and $g$, respectively, and which satisfy $f^{(q_n,\alpha)}_n \geq g^{(q_n,\alpha)}_n$ on $I$, whenever the scaling functions $\alpha_i$ are chosen according to
\begin{equation}\label{1eq38}
    0\leq \alpha_i(x)\leq \min\left\{\frac{\phi(f-g,i)}{\Phi_n(f-g)},1 \right\},\quad  i\in \mathbb{N}_{N-1}\quad x\in I,,
\end{equation}
where $\phi(f-g,i) :=\min\limits_{x\in I}(f-g)(u_i(x))$, and $\Phi_n(f-g):=\max\limits_{x\in I}M_{n,q_n}(f-g)(x)$.
\end{corollary}
\begin{proof}
The corollary follows immediately from Theorem \ref{1th4.3} by choosing $f$ as $f-g$ and $g=0$.
\end{proof}

\begin{Example}
In this example, we illustrate Corollary \ref{1coro4.4.1}. To this end, let $f(x) :=\sin(\pi x)$ and $g(x) :=-(2x-1)^2$, $x\in I:=[0,1]$. Let $\Delta:=\{0, \frac{1}{3}, \frac{2}{3}, 1\}$ and choose $\alpha_1(x):=0.6/(1+\exp(-8x))$, $\alpha_2(x):=0.6/(1+\exp(-7x))$,  and $\alpha_3(x):=0.6/(1+x^2)$. Further, let $q_n=\frac{2}{\pi}\arctan(n)$, $n\in \mathbb{N}$. Then $f$ and $g$, the  scaling functions $\alpha_i$, and $\{q_n\}_{n=1}^{\infty}$ satisfy the required conditions in Corollary \ref{1coro4.4.1}. (See, also Figure \ref{4fig} (a).)
Figure \ref{4fig} (b) depicts the quantum MKZ-fractal functions $f^{(q_2,\alpha)}_2 \geq g^{(q_2,\alpha)}_2$ on $I$. 
\end{Example}
\begin{figure}[ht!]
\centering
\begin{minipage}{0.5\textwidth}
\epsfig{file=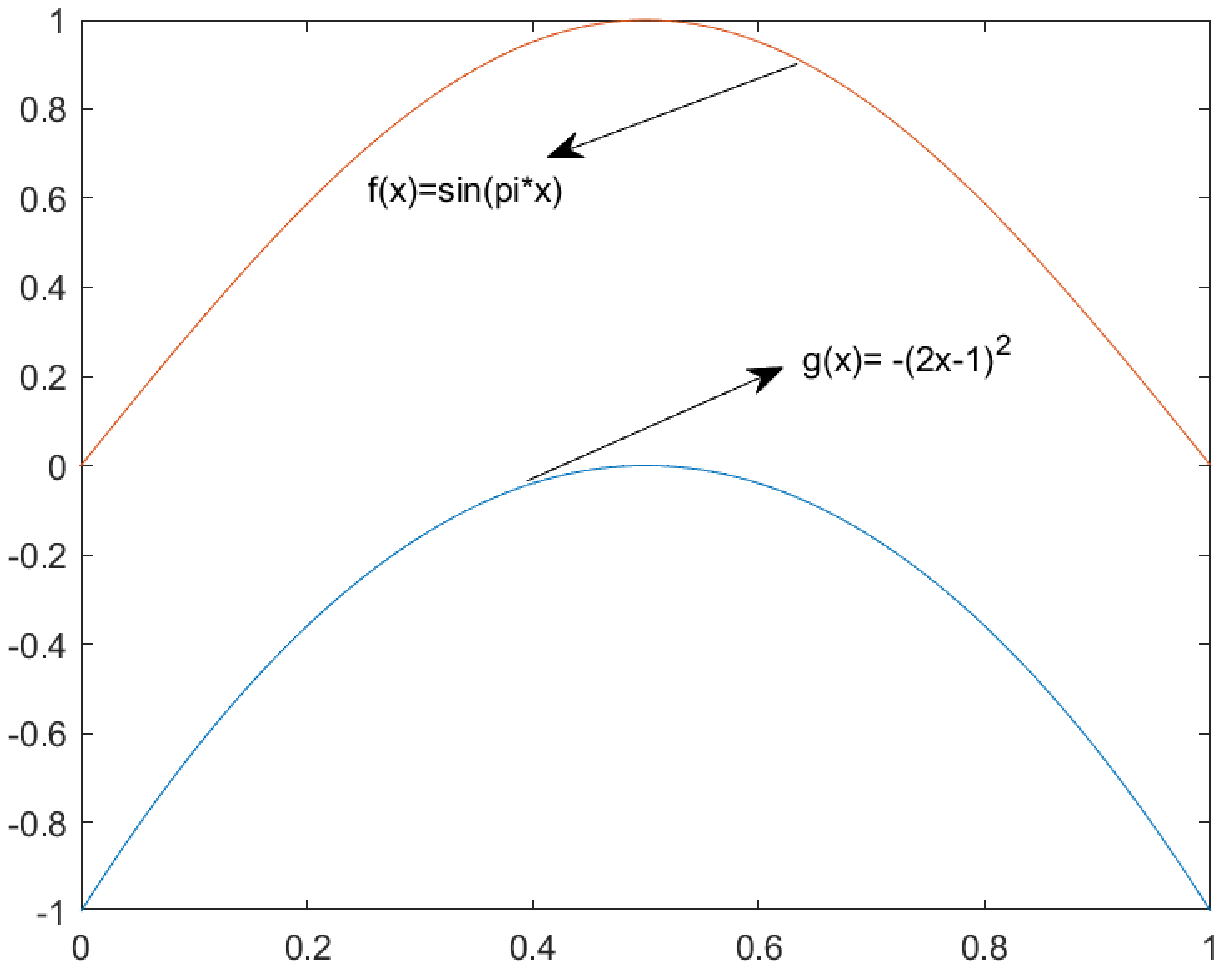,scale=0.4}\\
\centering{(a) $ f > g$ }
\end{minipage}\hfill
\begin{minipage}{0.5\textwidth}
\epsfig{file=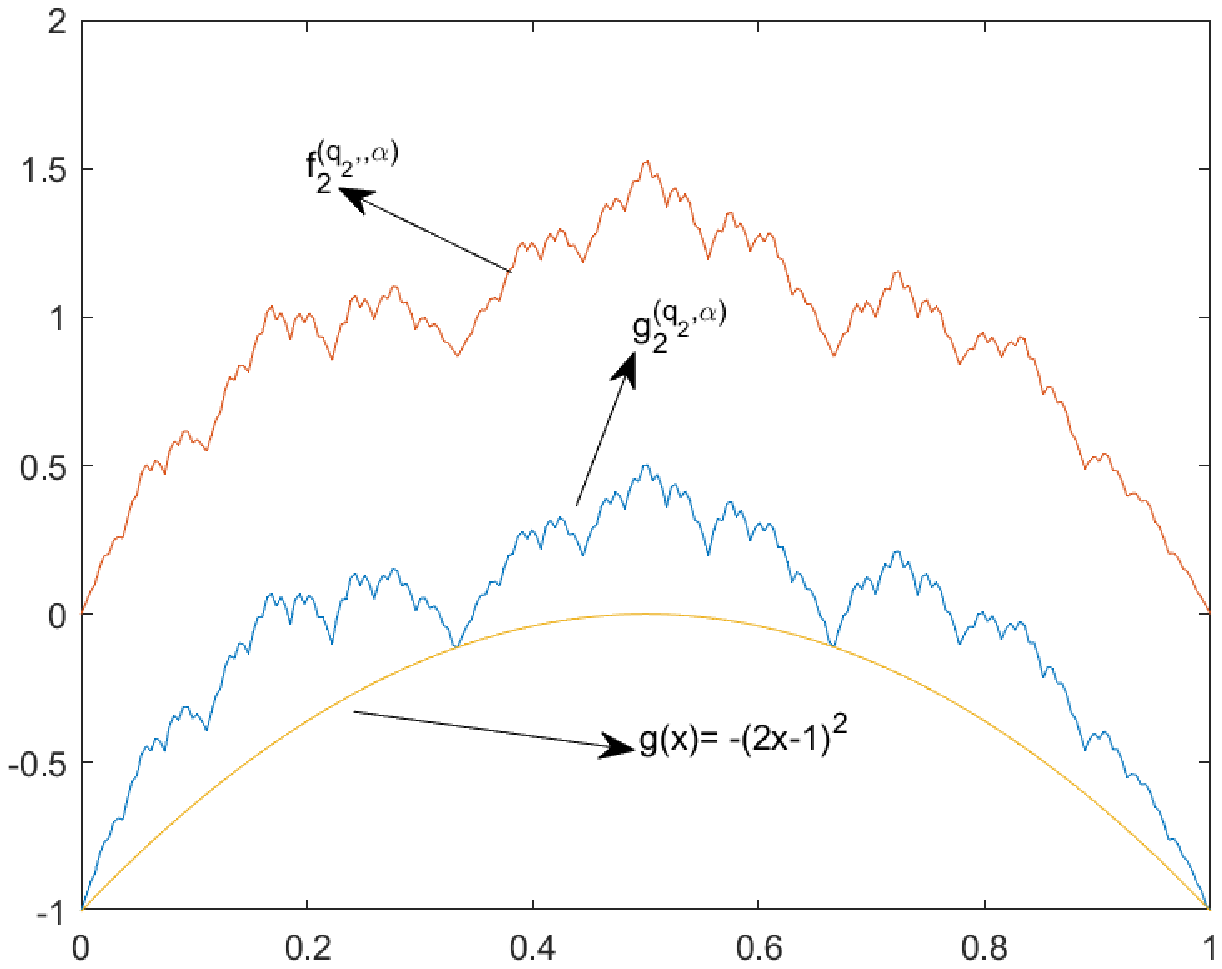,scale=0.4}\\
\centering{(b) $ f_2^{(q_2, \alpha)} > g_2^{(q_2, \alpha)}$ }
\end{minipage}\hfill
\caption{MKZ-fractal functions preserving positivity.}
\label{4fig}
\end{figure}

\begin{theorem}\label{monoth1}
Let $f\in C(I)$ be convex and $\alpha_i(x)\geq 0$ for all $x\in I,$ $i\in \mathbb{N}_N.$  Then $f_n^{(q,\alpha)}(x)\leq f(x)$ for all $x\in I$, and the sequence $\{f_n^{(q,\alpha)}(x)\}_{n=1}^{\infty}$ is non-increasing for each $x\in I$.
\end{theorem}
\begin{proof}
Using the functional equation \eqref{1eq17} of $f_n^{(q,\alpha)}$ and keeping $f(x)- M_{n}f(x)\geq0$ [Theorem 3.2, \cite{Trif}] in mind,
we conclude that for $x\in I$, $i\in \mathbb{N}_N$
\begin{align}\label{monotheq1}
 (f_n^{(q,\alpha)}- f)(u_i(x)) & = \alpha_i(x) (f^{\alpha}_n(x)-f(x)) +\alpha_i(f(x)- M_{n}f(x))\nonumber \\
 & \leq \alpha_i(x) (f^{\alpha}_n(x)-f(x)).
\end{align}
By means of \eqref{monotheq1}, we conclude that $f^{\alpha}_n(x)-f(x)\leq 0$ ensuring that $(f_n^{(q,\alpha)}- f)(u_i(x))\leq 0$. As the fractal function $f^{(q,\alpha)}_n$ is  constructed iteratively, we obtain $f_n^{(q,\alpha)}(x)\leq f(x)$ for all $x\in I$.\\ 
 From \eqref{1eq17}, the functional equations of $f_{n+1}^{(q,\alpha)}$ and $f_n^{(q,\alpha)}$, respectively, are
\begin{align*}
    f^{(q,\alpha)}_{n+1}( u_i(x)) & = f(u_i(x))+ \alpha _i(x)(f^{(q,\alpha)}_{n+1}(x) - M_{n+1,q}f(x)), \;\; x \in I ,\\
    f^{(q,\alpha)}_n( u_i(x)) & = f(u_i(x))+ \alpha _i(x)(f^{(q,\alpha)}_n(x) - M_{n,q}f(x)), \;\; x \in I. 
\end{align*}
Note that both $f_{n+1}^{(q,\alpha)}$ and $f_n^{(q,\alpha)}$ join up at the interpolation data points and that subsequent values are generated iteratively from the same data. Taking their difference and using the fact that $M_{n+1}f(x)- M_{n}f(x)\leq 0$ for all $x\in I$ from [Theorem 3.3, \cite{Trif}], we obtain that, for all $x\in I$, and $i\in \mathbb{N}_N$,
\begin{align}\label{monoeq2}
 (f_{n+1}^{(q,\alpha)}-f_{n}^{(q,\alpha)})(u_i(x)) & \leq \alpha_i(x)(f_{n+1}^{(q,\alpha)}-f_{n}^{(q,\alpha)})(x) +  \alpha_i(x)(M_{n+1}f- M_{n}f)(x)\nonumber\\
 &\leq \alpha_i(x)(f_{n+1}^{(q,\alpha)}-f_{n}^{(q,\alpha)})(x).
\end{align}
As the right hand side of \eqref{monoeq2} is zero at the first iteration, it is ensured that  $f_{n+1}^{(q,\alpha)}(x)\leq f_{n}^{(q,\alpha)}(x)$ for all $x\in I$.
\end{proof}
\begin{remark}
Using the hypotheses of Theorem \ref{1th3.1} and Theorem \ref{monoth1}, we can construct a non-increasing sequence of positive quantum MKZ-fractal functions converging to $f\in C(I)$  provided $f$ is a convex and non-negative.
\end{remark}
\section{Approximation with quantum MKZ-fractal M\"untz polynomials}\label{1sec5}
 Let $\Lambda := \{\lambda_i\}_{i=1}^{\infty}$, with $\lambda_i\neq \lambda_j$ if $i\neq j$, $\lambda_i >0$, and $\lambda_0:=0$. The set of real-valued monomials
 \begin{center}
    $\Lambda_m = \{ x^{\lambda_0},x^{\lambda_1},...,x^{\lambda_m}\}$  
 \end{center}
is called a finite  M\"untz system.
The linear space
\[
M_m(\Lambda) := \text{span}(\Lambda_m)=\text{span}\{ x^{\lambda_0},x^{\lambda_1},...,x^{\lambda_m}\}
\]
is known as a (finite) M\"untz space and 
\[
M(\Lambda):=\bigcup_{m=0}^{\infty}M_m(\Lambda)
\]
is referred to as the M\"untz system corresponding to $\Lambda$.
\begin{definition}
Let $I:=[0,1]$ and let $\Delta := \{x_1,....,x_N \}$ be a partition of $I$ satisfying  $0=x_1 <..... < x_N=1$. Suppose $\alpha :=(\alpha_1,.....,\alpha_{N-1})$ where $\alpha_i$ is a bounded function on $I$ with $\| \alpha_i\|_{\infty}<1$.

A quantum MKZ-fractal M\"untz polynomial is finite linear combination of functions $(x^{\lambda_i})^{(q,\alpha)}_{n}$, $\lambda_i\in \Lambda$, $i\in \mathbb{N}$, where 
\[
(x^{\lambda_i})^{(q,\alpha)}_{n} := \mathcal{F} ^{(q,\alpha)}_{\Delta,M_{n,q}}(x^{\lambda_i})=\mathcal{F}_n^{(q,\alpha)}(x^{\lambda_i}).
\]
 is a quantum MKZ-fractal M\"untz monomial.
\end{definition}
For $x^{\lambda_j}\in  C[0,1]$,  we have that 
\begin{align*}\label{1eq43}
\mathcal{F}_n^{(q,\alpha)}(x^{\lambda_j}) &= x^{\lambda_j}+\sum_{i=1}^{N-1}\alpha _i (u_i^{-1}(x))(\mathcal{F}_n^{(q,\alpha)}(u_i^{-1}(x))^{\lambda_j}) - \\ & M_{n,q}(u_i^{-1}(x))^{\lambda_j})\chi_{u_i(I)}(x), \quad x\in I_i, \;i\in \mathbb{N}_{N-1},n\in \mathbb{N}_{N-1}.
 \end{align*} 
 Let $\tilde{K} :=\{ (x^{\lambda_i})^{(q,\alpha)}_{n}: i,n\in \mathbb{N}\}$. Then the set $M(\Lambda)=\text{span}(\tilde{K})$ is termed the quantum MKZ-fractal M\"untz space associated with $\Lambda$.
 
\begin{theorem}(Fractal version of the full M\"untz theorem in $C[0,1])$\label{1th5.1}
Let  $\Delta :=\{x_1,..., x_N\}$ be a partition of $I = [0, 1]$ satisfying $0 = x_1 < ..< x_N = 1,$ and let $\{q_n\}_{n=1}^{\infty}$ be a sequence in (0,1] such that $\lim\limits_{n\to \infty} q_n =1$. Let
$\alpha :=(\alpha_1,.....,\alpha_{N-1})$ where $\alpha_i$ is a bounded function on $I$ with $\| \alpha_i\|_{\infty}<1$, $i\in \mathbb{N}_{N-1}$. Further, let $\Lambda=\{{\lambda_i}\}_{i=1}^{\infty}$ be a sequence of distinct positive real numbers. Then, the set 
\[
S :=\bigcup _{n=1}^{\infty} \bigcup _{m=1}^{\infty} \Span \{1,(x^{\lambda_1})^{(q_n,\alpha)}_{n},\ldots, (x^{\lambda_m})^{(q_n,\alpha)}_{n} \}
\]
 is dense in $C[0, 1]$ with respect to the sup norm  provided that
\begin{equation}\label{1eq44}
    \sum_{i=1}^{\infty}\frac{\lambda_i}{\lambda_i^2+1}=\infty.
\end{equation}
\end{theorem}
\begin{proof}
Let $\epsilon>0$ and $f\in C[0,1]$ be given. Then, by the classical M\"untz theorem \cite{Borwein_Erdelyi}, we know that $\Span \bigcup\limits_{m=0}^\infty\{1,x^{\lambda_1},\ldots, x^{\lambda_m}\}$ is dense in $C[0,1]$ with respect to sup norm  if $ \sum\limits_{i=1}^{\infty}\frac{\lambda_i}{\lambda_i^2+1}$, where $\{\lambda_i\}_{i=1}^{\infty}$ is a sequence distinct positive real numbers. Hence, there exists a Müntz polynomial $p(x)=\sum\limits_{s=0}^{l}a_sx^{\lambda_s},a_s\in\mathbb{R}$, such that 
\begin{equation}\label{1eq45}
    \|f-p\|_{\infty}<\frac{\epsilon}{2}.
\end{equation}
Since $p$ is continuous,   $\|\mathcal{F}_n^{(q_n,\alpha)}(p)-p\|_{\infty} \to 0$ as $n \to \infty$  by Theorem \ref{1th3.1}. Therefore, there exists a natural number $n_1$ such that 
\begin{equation}\label{1eq46}
    \|p_n^{(q_n,\alpha)}-p\|_{\infty}<\frac{\epsilon }{2},\quad \forall n\geq n_1,
\end{equation}
where $\mathcal{F}_n^{(q,\alpha)}(p)=p_n^{(q,\alpha)}=\sum\limits_{s=0}^{l}a_s(x^{\lambda_s})_{n}^{(q,\alpha)}$.
Now, by (\ref{1eq45}) and (\ref{1eq46}),
\begin{equation*}\label{1eq47}
        \|p_n^{(q_n,\alpha)}-f\|_{\infty} \leq \|p_n^{(q_n,\alpha)}-f\|_{\infty}+\|f-p\|_{\infty}
        <\frac{\epsilon}{2}+\frac{\epsilon}{2}=\epsilon, \quad \forall n\geq n_1.
\end{equation*}
Hence, there exists a sequence of quantum MKZ-fractal M\"untz polynomials converging to $f$ in sup norm and  thus $ S $
 is dense in $C[0,1]$.
\end{proof}
Using arguments similar to those in the proof of Theorem \ref{1th5.1}, we can prove Theorem \ref{1th5.2} using the classical M\"untz second theorem, see for instance  \cite{Borwein_Erdelyi,Cheney}.
\begin{theorem}(Fractal version of M\"untz second theorem in $C[0,1])$\label{1th5.2}
Let  $\Delta :=\{x_1,..., x_N\}$ be partition of $I = [0, 1]$ satisfying $0 = x_1 <....< x_N = 1$ and let $\{q_n\}_{n=1}^{\infty}$ be a sequence in $(0,1]$ such that $\lim\limits_{n\to \infty} q_n =1$. Let
$\alpha :=(\alpha_1,.....,\alpha_{N-1})$ be a scaling function where $\alpha_i$ is a bounded function on $I$ with $\| \alpha_i\|_{\infty}<1$, $i\in \mathbb{N}_{N-1}$ . Let $\{{\lambda_i}\}_{i=1}^{\infty}$ be a sequence of distinct positive real numbers such that $ \inf\limits_{i\geq 1}\lambda_i >0 $. Then,  $ S $
 is dense in $C[0, 1]$ with respect to the sup norm  provided that 
 \[
 \sum\limits_{i=1}^{\infty}\frac{1}{\lambda_i}=\infty.
 \]
\end{theorem}
\begin{proof}
Let $f\in C[0,1]$ and $\epsilon >0$. Then, the classical M\"untz  theorem \cite{Borwein_Erdelyi} ensures the existence of M\"untz polynomial $p(x)$ such that $\|p-f\|_{\infty}<\cfrac{\epsilon}{2}$, and from \eqref{1eq46}, there exists a natural number $n_1$ such that $\|p_n^{(q_n,\alpha)}-p\|_{\infty}<\frac{\epsilon }{2},\quad \forall n\geq n_1$. These two inequalities imply $\|p_n^{(q_n,\alpha)}-f\|_{\infty}<\epsilon,\quad \forall n\geq n_1$. Hence $S$ is dense in $C[0, 1]$ with respect to the sup norm.
\end{proof}
\begin{theorem}\label{1th5.3}
Let $\Delta:=\{x_1,..., x_N\}$ where $0 = x_1 < .... < x_N = 1$ be a partition of $I = [0,1]$ and let $\{q_n\}_{n=1}^{\infty}$ be a sequence in $(0,1]$ such that $\lim\limits_{n\to \infty} q_n =1$. Let $\alpha :=(\alpha_1,.....,\alpha_{N-1})$ be a scaling function where each $\alpha_i$ is a bounded function on $I$ with $\| \alpha_i\|_{\infty}<1$, $i\in \mathbb{N}_{N-1}$. If $S :=\{f_s: s\in \mathbb{N}\}$ is dense in  $C[0,1]$, then 
\[
\bigcup_{n=1}^{\infty} \Span \{ \mathcal{F}_n^{(q_n,\alpha)}(f_s): s\in \mathbb{N}\}
\]
is also dense in $C[0,1]$.
\end{theorem}
\begin{proof}
Let $f \in C[0,1]$ and $\epsilon>0$ be given. By density of the set $S$ in  $ C[0,1]$, there exists a polynomial of the form $p(x)=\sum\limits_{s=0}^{l}a_sx^{\lambda_s},a_s\in\mathbb{R}$, such that
$ \|f-p\|_{\infty}<\frac{\epsilon}{2}$.  Using Theorem  \ref{1th3.1}, we get  $ \|p_n^{(q_n,\alpha)}-p\|_{\infty}<\frac{\epsilon }{2},\quad \forall n\geq n_1$. Finally, these two inequalities result in $ \|p_n^{(q_n,\alpha)}-p\|_{\infty}<\epsilon,\quad \forall n\geq n_1$. Therefore, $S$ is dense in  $C[0,1]$.
\end{proof}
\section{Approximation by  of MKZ-fractal functions}
In this section, we investigate some approximation-theoretic properties of MKZ-fractal functions and derive conditions for such functions to belong to the Lebesgue spaces $L^p$ with $p\geq 1$.
\subsection{MKZ-fractal approximation and integral MKZ-fractal functions}
If we set $q=1$ in \eqref{1eq15}, then the $q$-MKZ-series $M_{n,q}f$ becomes the classical MKZ-series $M_{n}f$,  which is also known as the MKZ-series of $f\in C(I)$. (See, for instance, \cite{Trif2,Trif}.) The MKZ-series is given by the following expression:
\begin{equation}\label{mkz}
\begin{split}
&M_{n}f(x) := P_n(x) \sum _{k=0}^{\infty} \binom{k+n}{k}(x-x_1)^k f(x_1 +(x_N-x_1) \frac{k}{k+n});\;x_1\leq x<x_N,\\
& M_{n}f(x_N):=f(x_N),
\end{split}
\end{equation}
with 
\[
P_n(x) :=\dfrac{ (x_N-x)^{n+1}} {(x_N-x_1)^{n+1}}.
\]
$M_{n}f$ satisfies the endpoint interpolation conditions
\begin{equation}
M_{n}f(x_0)=f(x_1)\quad\text{and}\quad M_{n}f(x_N)=f(x_N).
\end{equation}
The IFS with maps given by \eqref{1eq11}, where the base function is taken to be $b:=M_{n}f$, determines an $\alpha$-fractal function 
\[
f^{\alpha}_{n} := \mathcal{F} ^{\alpha}_{\Delta,b}(f),
\]
termed an MKZ $\alpha$-fractal function associated with $f \in C(I)$. Furthermore, this MKZ $\alpha$-fractal function satisfies the functional equation
  \begin{equation}\label{eq003}
 f^{\alpha}_n = f+\sum_{i=1}^{N-1}(\alpha _i\circ u_i^{-1})(f^{\alpha}_n\circ u_i^{-1})- (M_{n}f)\circ (u_i^{-1})\chi_{u_i(I)}
 \end{equation} 
on $I$.

It is straight-forward to establish the estimate 
 \begin{align}\label{eq22}
     \|f_n^{\alpha}-f\|_{\infty}\leq \frac{\| \alpha\|_{\infty}}{1-\| \alpha\|_{\infty}}\|f-M_nf\|_{\infty}.
 \end{align}
Even if $\alpha \neq {0}$, we can get the convergence of $f_n^{\alpha}$ to $f$ as $n\to \infty$ by the following corollaries.
 \begin{corollary}
Let $f\in C(I)$ and $\alpha \ne {0}$. Then  $f_n^{\alpha}$  converges  to $f$ 
as $ n \to \infty$, and
 \begin{align*}
 \|f^{\alpha}_n-f\|_{\infty} \leq  \frac{31}{27} \,\omega \left(f,\frac{1}{\sqrt{n}}\right) \frac{\|\alpha\|_{\infty}}{1-\|\alpha\|_{\infty}},\quad n\in\N,
  \end{align*}
where $\omega$ denotes the usual modulus of continuity of a function $f$.
 \end{corollary}
 \begin{proof}
  Using  Corollary 2.3 in reference \cite{Lupas_Muller},
  \begin{center}
     $\|f-M_{n}f\|_{\infty} \leq \frac{31}{27}\omega (f,\frac{1}{\sqrt{n}})$.
  \end{center}
Together with \eqref{eq22}, we get the required estimate.
 \end{proof}
 
 \begin{corollary}
 Let $f\in C^{1}(I)$ and $\alpha \ne {0}$. Then the uniform error between the MKZ $\alpha$-fractal functions $f_n^{\alpha}$ and its germ $f$ is given by  
 \begin{align}\label{eq007}
 \|f^{\alpha}_n-f\|_{\infty} \leq  \frac{2(2+3\sqrt{3})}{27\sqrt{n}}\,\omega \left(f,\frac{1}{\sqrt{n}}\right) \frac{\|\alpha\|_{\infty}}{1-\|\alpha\|_{\infty}},\quad \forall\,n\in \N,
  \end{align}
  where $\omega$ denotes the modulus of continuity of $f$.
 \end{corollary}
 \begin{proof}
 Corollary 2.5 in reference \cite{Lupas_Muller} yields 
 \[
 \|M_{n}f-f\|_{\infty} \leq \frac{2(2+3\sqrt{3})}{27\sqrt{n}}\,\omega \left(f',\frac{1}{\sqrt{n}}\right),\quad n\in\N.
 \]
Using this estimate in \eqref{eq22}, we obtain \eqref{eq007}. In this case $f_n^{\alpha}$ converges to $f$ at a faster rate than for an $f\in C(I)$.
 \end{proof}
 
 \begin{definition}
Let  $A\in \mathbb{R}$, and $0<\beta \leq 1$. Then  $\Lip_A\beta$ is defined as the set of all functions satisfying 
\[
|f(x_2)-f(x_1)|\leq A|x_2-x_1|^{\beta}, \quad\forall\, x_1,x_2 \in [x_0,x_N].
\]
Such functions are also called uniformly h\"olderian with exponent $\beta$.
\end{definition}

 \begin{corollary}
 Let $f\in C^{1}(I)$ with $f' \in \Lip_M\beta$ for $0< \beta \leq 1$. Further, let $M > \R$. Then the fractal functions $f_n^{\alpha}$  defined in \eqref{eq003} satisfies the following estimate:
 \begin{align}
 \|f^{\alpha}_n-f\|_{\infty} \leq  \frac{2(2+3\sqrt{3})}{27\sqrt{n}} n^{\frac{-(\beta+1)}{2}} \frac{\|\alpha\|_{\infty}}{1-\|\alpha\|_{\infty}},\quad \forall\, n\in\N.
  \end{align}
 \end{corollary}
 \begin{proof}
  It is known by Corollary 2.6 of \cite {Lupas_Muller} that  
  $$\|M_{n}f-f\|_{\infty} \leq \frac{2(2+3\sqrt{3})}{27\sqrt{n}} n^{\frac{-(\beta+1)}{2}}.$$ 
  Employing the above estimate  in \eqref{eq22}, we obtain the required inequality. Thus, in this case, the convergence rate of $f_n^{\alpha}$ to the germ function $f$ is faster than that for $f\in C(I)$ or $f\in C^1(I)$.
\end{proof}

\begin{proposition}[Cf. \cite{Trif2}]\label{prop1}
Given a continuous function $f:[x_0,x_N]\to \mathbb{R}$, it holds that 
\[
f\in \Lip_A \beta \quad\Longleftrightarrow\quad M_nf\in \Lip_A \beta,\quad \forall\,n \in\N,
\]
where $(M_n)_{n\geq 1}$ is the sequence of MKZ operators defined in \eqref{mkz}.
\end{proposition}

 \begin{theorem}
Let $f\in C(I)$  be uniformly h\"olderian with exponent $\beta \in (0,1]$,  satisfying $M_nf(x_0)=f(x_0)$ and $M_nf(x_N)=f(x_N)$. Suppose $\Delta:=\{x_0,\dotsc,x_N\}$ is a uniform  partition of $I$ satisfying $x_i-x_{i-1}=\lambda < 1$, for $i\in \mathbb{N}_N$, and $\alpha=(\alpha_1,\alpha_2,\dotsc,\alpha_N)\in {(-1,1)}^{N-1}$. 

Consider the IFS $\mathcal{I}_n:=\{I;(u_i(x),v_i(x,y)): i\in \mathbb{N}_N \}$, where $u_i:I \to I_i $ is an affine map with 
\[
u_i(x):=\lambda(x+{i-1}), 
\]
and 
\[
v_i(x,y):=\alpha_iy+f(u_i(x))-\alpha_i M_nf(x).
\]
 Further, assume that the set of interpolation points $\{(x_j, f(x_j)) : j\in \N_N\cup\{0\}\}$ is not collinear. Let 
\[
\gamma:=\sum_{i=1}^{N}|\alpha_i|,\quad\text{where $\alpha_i\ne 0$, for all $i\in \N_N$}.
\] 
Then, the box dimension of the graph $G=\{(x,f^{\alpha}(x)):x\in I\}$ of $f^{\alpha}$ has the following bounds:
\begin{align}
    \begin{cases}
      1\leq \dim_B G\leq 2-\beta,& \text{for}\;\; \gamma\leq 1;\\
      1\leq \dim_B G\leq 2-\beta+\log_N\gamma,& \text{for $\gamma > 1$ with $\gamma N^{\beta-1}\leq 1$};\\
      1\leq \dim_B G\leq 1+\log_N\gamma,& \text{for $\gamma > 1$ with $\gamma N^{\beta-1}> 1$};\\
       \dim_B G\geq 1+\log_N\gamma,& \text{for $\gamma > 1$ with $\beta=1$}.
    \end{cases}
\end{align}
\end{theorem}
 \begin{proof}
By Proposition \ref{prop1}, $ f\in \Lip_A\beta$ implies $M_nf \in \Lip_A\beta$. Therefore, $M_nf$ is also h\"olderian with exponent $\beta$. The statements now follow from \cite[Theorem 3.1]{Akhtar_Prasad_Navascues} for H\"older functions.
 \end{proof}
 \subsection{{MKZ $\alpha$-fractal functions in $L^p$-spaces}}
 %

Given $f\in L^p(I)$, $p \geq 1$, where $I:=[0,1]$,  the integral MKZ operator  $\wh{M}_n$\cite{Muller} is defined as
 \begin{align}\label{eqlp00}
  \wh{M}_nf(x) := \int_0^1 H_n(x,s)f(s)ds,
 \end{align}
 where 
\[
H_n(x,s):=\sum_{k=0}^{\infty} \wh{m}_{nk}(x)\chi_{I_k}(t),
\]
\be\label{6.9}
\wh{m}_{nk}(x):=(n+1)\binom{k+n+1}{k}x^k(1-x)^n, 
\ee
and
\be\label{6.10}
I_k:=\left[\frac{k}{k+n},\frac{k+1}{k+n+1}\right],
\ee
with $\chi_{I_k}$ denoting the characteristic function on $I_k$.
 
In the following theorem, we define MKZ $\alpha$-fractal functions in $L^p$-spaces using as a base function $\wh{M}_n$. 
 \begin{theorem}\label{thlp}
Let $f\in L^p(I)$ with $p \geq 1$ and $I:=[0,1]$. Suppose that $\Delta:=\{0=x_0< x_2<\dotsc < x_N=1\}$ is a partition of $I$. Denote by $I_i:=[x_i,x_{i+1})$, $i\in \mathbb{N}_{N-1}$, and $I_N:=[x_{N-1},x_{N}]$ the subintervals induced by $\Delta$. 

Define affine maps $u_i:[0,1]\to I_i$ on $I$ by $u_i(x) :=a_ix+b_i$. Assume that \begin{gather*}
u_i(x_0)=x_{i-1},  \quad u_i(x_N-)=x_{i}, \quad i \in \mathbb{N}_{N-1},\\
u_N(x_0)=x_{N-1}, \quad u_N(x_N)=x_N. 
\end{gather*}
Suppose  $\alpha_i \in {L}^{\infty}(I)$ for $i\in \mathbb{N}_{N}$. Define an RB-operator $T:L^p(I)\to L^p(I)$ by
\begin{align}
Tg :=f + \sum_{i=1}^N (\alpha_i \circ u_i^{-1})(g-\wh{M}_nf)\circ u_i^{-1} \, \chi_{u_i(I)}.
\end{align}
Assume that the vector of scaling functions $\alpha:= (\alpha_1, \ldots, \alpha_N)\in (L^\infty(I))^N$ satisfies
  \begin{align}
\Lambda:=\left(\sum_{i\in \mathbb{N}_{N}}a_i\|\alpha_i\|_{\infty}^p\right)^{\frac{1}{p}}<1,   \quad\text{for $1\leq p <\infty$}.
  \end{align}
Then $T$ is a contraction map on $L^p(I)$ with Lipschitz constant $\Lambda$ and its unique fixed point $f_n^{\alpha}\in L^p(I)$ satisfies
\begin{align}\label{eq001}
  f_n^{\alpha} = f + \sum_{i=1}^N (\alpha_i\circ u_i^{-1})(f_n^{\alpha}-\wh{M}_nf)\circ u_i^{-1}\, \chi_{u_i(I)}.
\end{align}
Furthermore, the mapping $\mathcal{F}_{\Delta,\hat{M}_n}^{\alpha}: L^p(I) \to  L^p(I)$ with 
\[
\mathcal{F}_{\Delta,\wh{M}_n}^{\alpha}(f)=f_{n}^{\alpha}
\]
is a bounded linear operator.
 \end{theorem}
 \begin{proof}
Using the same arguments as in the proof of \cite[Theorem 2.1]{Vishwanathan_Navascues_Chand}, we can easily verify that $T$ is contractive. The remaining part follows using Theorem 3.4 in \cite{Chand_Jha_Navascues}.
 \end{proof}
 \begin{theorem}
 Let $f\in L^p(I)$, $p\in [1,\infty)$. The MKZ fractal functions $f_n^{\alpha}$, $ n\in \mathbb{N}$, defined in Theorem \ref{thlp} converge in $L^p$-norm to $f$ as $n\to \infty$ and satisfies the estimate
 \begin{align}\label{eq002}
   \|f_n^{\alpha}-f\|_{p} \leq C\frac{\Lambda}{1-\Lambda} \omega_{1,p}\left(f,\frac{1}{\sqrt{n}}\right),
 \end{align}
 where $C > 0$ is a constant independent of $f$ and $p$, 
 \[
 \omega_{1,p}(f,t):=\sup_{0<h\leq t}\|f(.+h)-f(.)\|_{p,I_h}, 
 \]
 where $\| .\|_{p,I_h}$ is the $L^p$ norm taken over the interval $I_h=[0,1-h]$.
 \end{theorem}
 \begin{proof}
 From \eqref{eq001}, it is easy to compute that
 \begin{align*}
    \begin{split}
       \|f_n^{\alpha}-f\|_{p}^p \leq&  \Lambda^p  \|f_n^{\alpha}-\hat{M}_nf\|_{p}^p.
       \end{split} 
 \end{align*}
     Equivalently,
        \begin{align*}
    \begin{split}
      \|f_n^{\alpha}-f\|_{p} \leq &\Lambda  \|f_n^{\alpha}-\hat{M}_nf\|_{p}\\
        \leq&  \Lambda \left\{ \|f_n^{\alpha}-f\|_{p} + \|f-\wh{M}_nf\|_{p} \right\},
    \end{split} 
 \end{align*}
 which gives
  \begin{align}\label{eq0012}
      \|f_n^{\alpha}-f\|_{p} \leq \frac{\Lambda}{1-\Lambda} \|f-\wh{M}_nf\|_{p}. 
  \end{align}
   This  ensures $\|f_n^{\alpha}-f\|_{p} \to 0$ as $n \to \infty$.
  
It is known from Theorem 3 in \cite{Muller} that $\|f-\wh{M}_nf\|_{p} \leq C \omega_{1,p}\left(f,\frac{1}{\sqrt{n}}\right)$, for some $C>0$. Using this result in \eqref{eq0012}, we obtain the required estimate.
 \end{proof}
 Now we give the MKZ-fractal version of the Full M\"untz Theorem for $L^p[0,1]$.
\begin{theorem}
Let $\{\lambda_i\}_{i=0}^{\infty}$ be the sequence of district real numbers such that $\lambda_i\geq -\frac1p,$ for all $i\in \mathbb{N}$, where $p\in [1,\infty)$ and $(x^{\lambda_i})_n^{\alpha}$ is defined as in \eqref{eq001}. Then the  MKZ--M\"untz space, 
\[
S:=\bigcup_{n=1}^{\infty} \bigcup_{m=1}^{\infty} \Span\{(x^{\lambda_0})_n^{\alpha},(x^{\lambda_1})_n^{\alpha},\dotsc,(x^{\lambda_m})_n^{\alpha}\},
\]
is dense in $L^p(I)$ provided that 
\begin{align*}
    \sum_{i=0}^{\infty} \frac{\lambda_i+\frac{1}{p}}{(\lambda_i+\frac{1}{p})^2+1}=\infty.
\end{align*}
\end{theorem}
\begin{proof}
Let $f\in L^p(I)$ and $\epsilon>0$ be given. By \cite{Borwein_Erdelyi},   it is known that the space $\bigcup\limits_{m=1}^{\infty} \Span\{x^{\lambda_0},x^{\lambda_1},\dotsc,x^{\lambda_m})\}$ is dense in $ L^p(I)$. Therefore, there exists a M\"untz polynomial $r(x):=\sum\limits_{i=1}^l d_i\, x^{\lambda_i} \in L^p(I)$, $d_i\in\R$, such that 
\begin{align*}
  \|f-r\|_p<\frac{\epsilon}{2}.  
\end{align*}
As  $r\in L^p(I)$, there exists a natural number $n_0$ such that
\begin{align}\label{eqlp3}
\begin{split}
 \|f-r_n^{\alpha}\|_p<&\frac{\epsilon}{2},\quad \text{for all}\; n\geq n_0.
     \end{split}
 \end{align}
 Using the linearity of  $\mathcal{F}_{\Delta,\hat{M}_n}^{\alpha}$, we can write  $r_n^{\alpha}(x)=\mathcal{F}_{\Delta,\hat{M}_n}^{\alpha}(r)(x)=\sum\limits_{i=0}^{l}d_i(x^{\lambda_i})_n^{\alpha}$ and, thus, the above inequality becomes
 \begin{align}\label{eqlp4}
 \left\|f-\sum_{i=0}^{l}d_i(x^{\lambda_i})_n^{\alpha}\right\|_p< \frac{\epsilon}{2},\quad\text{for all}\; n\geq n_0.
\end{align}
From \eqref{eqlp3} and \eqref{eqlp4}, we obtain  
\[
\left\|f-\sum_{i=0}^{l}d_i(x^{\lambda_i})_n^{\alpha}\right\|_p<\epsilon,\quad\text{for all $n\geq n_0$}.
\] 
Hence $S$ is dense in $L^p(I)$.
\end{proof}
\subsection{A remark about the case $0<p<1$}
Recall that for $0 < p < 1$, the Lebesgue spaces ${L}^p (I)$, $I\subset \R$ a compact interval, are $F$-spaces whose topology is induced by the complete translation invariant metric 
\[ 
d_p (g,h) := \n{g-h}_{p}^p = \int_I |g(x) - h(x)|^p dx.
\] 
(See, \cite[1.47]{Ru}.)

We will show by a counterexample that the above results do not extend to the case $0<p<1$. For this purpose, we need to quote a theorem of Orlicz's. 
\begin{theorem}[Orlicz's Theorem \cite{Lorentz}]\label{Orlicz}
Suppose that the kernel $K_n(x,t)$ is measurable in the square $\{(x,t)\in \R\times\R : a\leq x\leq b$, $a\leq t\leq b\}$ and that 
\begin{equation}
    \int_a^b|K_n(x,t)|dt\leq M,\quad\text{a.e. $x\in [a,b]$},
    \end{equation}
\begin{equation}
    \int_a^b|K_n(x,t)|dx\leq M,\quad\text{a.e. $t\in [a,b]$},
\end{equation}
with a constant $M$ and for all $n\in \N$. Then, for $f\in L^p[a,b]$, the singular integral 
\begin{equation}
    F_n(x)=\int_a^b K_n(x,t)f(t)dt
\end{equation}
exists for a.e. $x$ and is a function of the class $L^p[a,b]$. If in addition $F_n \to f$ strongly for all elements $f\in H$ of a set $H \subset L^p[a,b]$ which is everywhere dense in $L^p[a,b]$, then this is also true for any $f\in L^p[a,b]$:
\begin{equation}
    \|f-F_n\|=\left [ \int_a ^b|f(x)-F_n(x)|^p dx \right ]^{\frac{1}{p}} \to 0.
\end{equation}
\end{theorem}
Orlicz theorem is only true for $p\geq 1$. The following example shows that the above theorem does not hold for $0<p<1$:

\begin{Example}
Let $[a,b]:= I = [0,1]$, $K_n(x,t) :=c$, $c\ne 0$, and $f(t) :={t^{-1}}$. Then, $f\in L^p(I)$ for $0<p<1$ and $K_n(x,t)$ satisfy all required conditions of Orlicz's theorem, but
\begin{equation*}
 F_n(x) = \int_0^1 K_n(x,t)f(t)dt =\int_0^1 c \frac{1}{t}dt=\begin{cases} 
 +\infty, & c>0; \\
 -\infty, & c<0.
 \end{cases}
\end{equation*}
\end{Example}
For $p\geq 1$, M\"uller\cite{Muller} proved that the operator $\widehat{M}_n$ is well-defined and bounded using Orlicz's theorem \ref{Orlicz}. The above choices for $K_n$ and $f$ imply that Orlicz's theorem fails to prove the well-definiteness of the operator $\widehat{M}$ in $L^p$ for $0<p<1$. More precisely, take again
%
$f(t):={t}^{-1}$. As seen above, $f\in L^p$ for $0<p<1$.  By the positivity of $\widehat{m}_{nk}(x)$ (i.e., the series given below is a series of positive terms), we have
\begin{align*}
    \widehat{M}_nf(x) &=\int_0^1 H_n(x,t)f(t) dt= \sum_{k=0}^{\infty} \widehat{m}_{nk}(x)\int_{I_k}f(t) dt \geq  \widehat{m}_{n0}(x)\int_{I_0}f(t) dt\\
    &= (n+1)(1-x)^n \lim_{\varepsilon\to 0+} \int_\varepsilon^1 \frac{dt}{t} =\infty,
\end{align*}
where $\widehat{m}_{nk}$ and $I_k$ are given by \eqref{6.9} and \eqref{6.10}, respectively.
Hence, the integral MKZ operators  $\widehat{M}_n$ are  not well-defined on $L^p$ for $0<p<1$.

%
\section{Conclusions}\label{1sec6} 
In this paper, a novel class of fractal functions is introduced using quantum Meyer-K\"onig-Zeller (MKZ) functions as base functions in the $\alpha$-fractal interpolation procedure. For $f\in C(I)$,  we have constructed a sequence of quantum MKZ-fractal functions  that converges uniformly to $f$ without altering the scaling functions. The convergence and  shape  of  quantum  fractal approximants depend on  the variable $q\in(0,1]$ as well as the scaling functions.  For a given positive function $f\in C(I)$, we have generated  a sequence of positive MKZ-fractal functions that converges uniformly to $f$ provided the scaling functions satisfy certain growth restraints.  We have shown the existence of one-sided MKZ-fractal approximants for a given function and proved MKZ-fractal versions of M\"untz theorems  in  $C[0,1]$. Finally, we have investigated some approximation-theoretic properties of MKZ $\alpha$-fractal functions  in $C(I)$ and $L^p$-spaces.

\end{document}